\documentclass[a4paper]{article}
\usepackage{amsfonts, amsmath}
\usepackage{graphicx}
\usepackage{tikz}
\usepackage{float}
\usepackage{subfig}
\usepackage{commath}					
\usepackage{mathrsfs}					
\usepackage{color}
\usepackage{microtype}
\usepackage{hyperref}
\usepackage{amsthm}
\definecolor{dkgreen}{rgb}{0, 0.6, 0}
\newtheorem{theorem}{Theorem}

\title{Network-inspired versus Kozeny-Carman based permeability-porosity relations applied to Biot's poroelasticity model}
\author{Menel Rahrah\thanks{Delft Institute of Applied Mathematics, Delft University of Technology, Delft, The Netherlands. Emails: \{M.Rahrah, L.A.Lopezpena, F.J.Vermolen, B.J.Meulenbroek\}@tudelft.nl} \and Luis A. Lopez-Pe\~na\footnotemark[1] \and Fred Vermolen\footnotemark[1] \and Bernard Meulenbroek\footnotemark[1]}
\date{\today}

\begin{document}
\maketitle

  \makeatother
  \pagestyle{myheadings}
  \markright{M. Rahrah et al. \quad Permeability-porosity relations applied to Biot's model}

\begin{abstract}
Water injection in the aquifer induces deformations in the soil. These mechanical deformations give rise to a change in porosity and permeability, which results in non-linearity of the mathematical problem. Assuming that the deformations are very small, the model provided by Biot's theory of linear poroelasticity is used to determine the local displacement of the skeleton of a porous medium, as well as the fluid flow through the pores. In this continuum scale model, the Kozeny-Carman equation is commonly used to determine the permeability of the porous medium from the porosity. The Kozeny-Carman relation states that flow through the pores is possible at a certain location as long as the porosity is larger than zero at this location in the aquifer. However, from network models it is known that percolation thresholds exist, indicating that the permeability will be equal to zero if the porosity becomes smaller than these thresholds. In this paper, the relationship between permeability and porosity is investigated. A new permeability-porosity relation, based on the percolation theory, is derived and compared with the Kozeny-Carman relation. The strongest feature of the new approach is related to its capability to give a good description of the permeability in case of low porosities. However, with this network-inspired approach small values of the permeability are more likely to occur. Since we show that the solution of Biot's model converges to the solution of a saddle point problem for small time steps and low permeability, we need stabilisation in the finite element approximation.

\end{abstract}

\section{Introduction} \label{Intro}

For the description of different physical processes, such as consolidation, it is of a pivotal importance to have a valid estimation of permeability. The permeability of porous media is usually expressed as a function of some physical properties of the interconnected pore system such as porosity. Although it is natural to assume that the permeability depends on the porosity, it is not simple to formulate satisfactory theoretical models for the relation between them, mainly due to the complexity of the connected pore space. One of the most widely used permeability-porosity relationships is the Kozeny-Carman relation~\cite{CARM37,KOZE27}. This relation assumes that the connectivity of a porous space does not vary in time, either by assuming a pore space that stays fully connected or by taking the effective porosity initially and assuming no loss of connectivity. Therefore, the Kozeny-Carman relation assumes that flow through the porous medium is possible as long as the porosity is nonzero. Hence, this relation is not capable of predicting blocking of the flow if the porosity is too low in some parts of the porous medium leading to two or more disconnected regions. Moreover, it is empirically proven that the permeability decreases dramatically with decreasing porosity~\cite{BBE82}, indicating that the Kozeny-Carman relation is less accurate at low porosities. To improve the behaviour of this relation for small values of the porosity, Mavko and Nur~\cite{MANU97} incorporated a simple porosity adjustment into the Kozeny-Carman relation, by taking the percolation threshold into account. This approach resulted in a better prediction of the permeability for low porosities. However, in the work of Mavko and Nur the percolation threshold was chosen empirically to give a good fit for the experimental data. Another approach to take into account the percolation threshold was presented by Porter et al.~\cite{PRMD13}, based on the work of Koltermann and Gorelick~\cite{KOGO95}, who adapted the Kozeny-Carman equation to represent bimodal grain-size mixtures.

The Kozeny-Carman equation is based on only having spherical grains in the porous medium, whereas these grains can have various shapes. In this sense, the Kozeny-Carman relation represents a limit case. Another limit case is the assumption that the voids between the grains are represented by straight channels. In this study, we briefly introduce a new approach for the permeability that is derived on a micro-scale network model. At the pore-scale, this network model offers a detailed description of the porous medium~\cite{BEEW98}. The current modelling framework deals with the latter limit case, and can be used for general network topologies (such as rectangular, triangular and cubic arrangements of the channels). In contrast with the Kozeny-Carman equation, the new network-inspired approach yields that the permeability is nonzero only if the porosity is larger than a specific percolation threshold, that depends on the topology of the network. This means that the new approach may give a better prediction of the permeability in physics problems where abrupt changes in the porous medium occur resulting in a loss of connectivity of the connected pore space. In addition, the current model provides a computational framework to determine the percolation threshold for any given network with straight channels, which implies that there is no need for fitting on the basis of experiments if the network topology is known. This numerically determined threshold value agrees very well with the known values of the percolation thresholds from the literature. The new permeability-porosity approach is derived from percolation theory, which is a branch of probability that describes the effects of randomness arising from the porous media structure~\cite{BRHA57}.

In percolation theory, the nodes of the network are called sites and the edges are called bonds. This leads to two approaches: site percolation and bond percolation. In this paper, we consider the bond percolation approach, whose basic idea can be explained as follows. Consider a network such as shown in Fig.~\ref{Fig:Rectangular_network}. Whether a bond is open or closed depends on a certain probability and it is independent of the neighbouring bonds. If one bond is open and the nearest neighbours are closed, it is said that a 1-cluster is formed. If two adjacent bonds are open, they form a 2-cluster and so on. If the probability $p$ of a bond being open increases, then larger clusters are created. There exists a critical probability at which the first cluster that spans the entire network from the inlet to the outlet is obtained. For infinite networks, this critical probability is well defined and it is called the percolation threshold~$p_c$~\cite{BEEW98}.

\begin{figure}[!ht]
    \centering
    \includegraphics[width=0.27\textwidth]{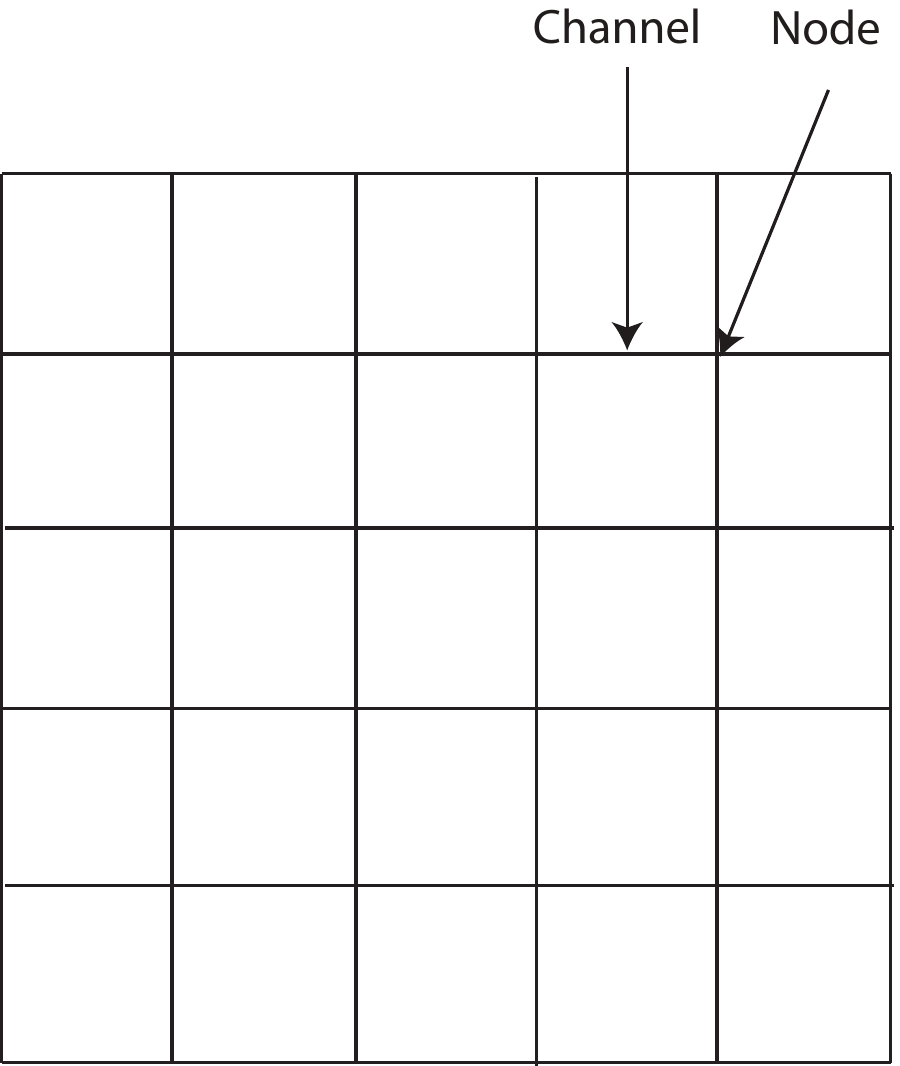}
    \caption{An illustration of a rectangular network.} \label{Fig:Rectangular_network}
\end{figure}

Some macroscopic properties of porous media are mainly determined by the connectivity of the pore system~\cite{BEEW98}. Hence, it is possible to find a relation between the permeability of a porous medium and the percolation properties using the network model. For instance, if each of the open bonds conducts a fluid and there exists a cluster that spans the network from the inlet to the outlet, then a volumetric flow is possible through the network. By adding more open bonds to the cluster, the volumetric flow through the network may increase. Since the number of open bonds represents the porosity of the network, it is possible to find a relation between the volumetric flow and the porosity~\cite{BALB87,WONG88}. In addition, Darcy's law gives a relationship between the permeability of the porous medium and the volumetric flow. As a result, a relation between the permeability and the porosity can be derived.

The applicability of these permeability-porosity relations will be demonstrated using two illustrative numerical examples. In these examples, flow of an incompressible fluid through a linearly elastic, saturated porous medium is modelled, using the classical theory of poroelasticity. Originally developed by Biot~\cite{BIOT41}, poroelasticity theory assumes a superposition of solid and fluid components and couples the mechanical deformation of a porous solid with fluid flow through its internal structure. This approach is valid in the infinitesimal deformation range. Poroelasticity problems exist widely in the real world, making this theory of a great interest due to its applicability in various branches of science and engineering (see e.g.~\cite{BBBG16,CHEN16,RAVE17,SPIE93,SDHH12}).

In order to investigate the difference between the Kozeny-Carman approach and the equation based on the percolation theory, the boundary conditions in both academic poroelasticity examples are chosen such that a decrease in porosity is realised in some parts of the computational domain. This decrease in porosity will lead in both relations to a decrease in the permeability of the porous medium. When using the finite element method to solve the poroelastic equations, it is well known that the numerical solution of these equations may exhibit nonphysical oscillations in the pressure field for low permeabilities and short time steps~\cite{HOL12,VEVE81}. In Sect.~\ref{Sec:Proof}, we will prove that under these conditions, the resulting finite element discretisation approaches saddle point problems. Hence, a considerate numerical methodology in terms of possible spurious oscillations is needed. Therefore, a Galerkin finite element method based on Taylor-Hood elements has been developed, combined with the stabilisation technique proposed in~\cite{AGLR08,RGHZ16}. Another stabilised finite element method is employed by Wan~\cite{WAN03}, Korsawe and Starke~\cite{KOST05} and Tchonkova et al.~\cite{TPS07}, based on the Galerkin least-squares method.

Poroelasticity problems have been attracting attention from the scientific computing community~\cite{HRGZ17,WXY14} (and references therein). Another numerical methods that are used to solve the poroelasticity equations are the finite volume method combined with a nonlinear multigrid method as adopted by Luo et al.~\cite{LRGO15}. In addition, stabilised finite difference methods using central differences on staggered grids are used by Gaspar et al.~\cite{GLV03,GLV06} to solve Biot's model. However, the numerical solution of poroelasticity equations has been traditionally associated with finite element methods~\cite{BRK17,PHWH07}. Furthermore, a monolithic approach for solving the quasi-static two-field poroelasticity equations is employed in this paper, which involves solving the coupled governing equations of flow and geomechanics simultaneously at every time step. Another approach that is widely used in coupling the flow and the mechanics in porous media is the fixed stress split method~\cite{BBNK17,MIWH13}. Hong et al.~\cite{HKLP19,HKLW19} applied the fixed stress splitting scheme to multiple-network poroelasticity systems.


\section{Governing equations} \label{Sec:Governing_equations}

The model provided by Biot's theory of linear poroelasticity with single-phase flow~\cite{BIOT41} is used in this study to determine the local displacement of the grains of a porous medium and the fluid flow through the pores, assuming that the deformations are very small. The fluid-saturated porous medium has a linearly elastic solid matrix and is saturated by an incompressible Newtonian fluid. Let $\Omega \subset \mathbb{R}^2$ denote the computational domain with boundary $\Gamma$, and $\mathbf{x} = (x, y) \in \Omega$. Furthermore, $t$ denotes time, belonging to a half-open time interval $I = (0, T]$, with $T > 0$. The initial boundary value problem for the consolidation process of an incompressible fluid flow in a deformable porous medium is stated as follows~\cite{AGLR08,WANG00}:
\begin{subequations}
  \begin{align}
    -\nabla \cdot \boldsymbol{\sigma}' + \nabla p = \mathbf{0}				&\mbox{\ on\ } \Omega \times I; \label{Eq:equilibrium} \\
    \frac{\partial}{\partial t} (\nabla \cdot \mathbf{u}) + \nabla \cdot \mathbf{v} = 0	&\mbox{\ on\ } \Omega \times I, \label{Eq:continuity}
  \end{align} \label{Eq:Biot}
\end{subequations}
where $\boldsymbol{\sigma}'$ and $\mathbf{v}$ are defined by the following equations
\begin{align}
  \boldsymbol{\sigma}' &= \lambda \mathrm{tr}(\boldsymbol{\varepsilon}) \mathbf{I} + 2\mu \boldsymbol{\varepsilon}; \label{Eq:constitutive} \\
  \boldsymbol{\varepsilon} &= \frac{1}{2}(\nabla \mathbf{u} + \nabla \mathbf{u}^T); \\
  \mathbf{v} &= -\frac{\kappa}{\eta} \nabla p. \label{Eq:Darcy}
\end{align}
In the above relations, $\boldsymbol{\sigma}'$ and $\boldsymbol{\varepsilon}$ denote the effective stress and strain tensors, $p$ the pore pressure, $\mathbf{u}$ the displacement vector, $\mathbf{v}$ Darcy's velocity, $\lambda$ and $\mu$ the Lam\'e coefficients; $\kappa$ the permeability of the porous medium and $\eta$ the fluid viscosity. The parameter values used in this paper are given in Table~\ref{Tab:MaterialValues}. In addition, appropriate boundary and initial conditions are specified in Sect.~\ref{Sec:Problem_formulation}.


\section{The permeability-porosity relations} \label{Sec:Porosity_permeability_relations}

In this study, we consider the spatial dependency of the porosity and the permeability of the porous medium. The porosity $\theta$ is computed from the displacement vector using the porosity-dilatation relation (see~\cite{RAVE18,TCH06})
\begin{equation}
  \theta(\mathbf{x}, t) = 1 - \frac{1 - \theta_0}{\exp(\nabla \cdot \mathbf{u})}, \label{Eq:porosity-dilatation}
\end{equation}
with $\theta_0$ the initial porosity. Subsequently, the permeability can be determined using the Kozeny-Carman equation~\cite{WAHS09}
\begin{equation}
  \kappa(\mathbf{x}, t) = \frac{d_s^2}{180} \frac{\theta(\mathbf{x}, t)^3}{(1 - \theta(\mathbf{x}, t))^2}, \label{Eq:KC}
\end{equation}
where $d_s$ is the mean grain size of the soil. The Kozeny-Carman relation assumes that the permeability becomes zero if and only if the porosity also becomes zero. A new approach for the relation between the porosity and the permeability is inspired by the fluid flow through a network, where the fluid flows into the edges (channels) of the network. The network-inspired relation takes into account that a certain number of the channels are removed randomly, therefore the fluid can not flow through those particular channels causing an alteration in the permeability and the porosity of the network. The number of removed channels increases from $1\%$ of the total number of channels to $100\%$. For a certain number of removed channels, there are no connected paths anymore between the inlet and the outlet. In this case, the fluid will stop flowing and the permeability will be expected to become zero. For an arbitrary network topology, the network-inspired relation states:
\begin{equation}
  \kappa(\mathbf{x}, t) = \begin{cases} 0 & \theta < \hat{\theta} \\ \frac{\theta - \hat{\theta}}{\theta_0 - \hat{\theta}}\kappa_0 & \theta \geq \hat{\theta} \end{cases}, \label{Eq:NI}
\end{equation}
where $\kappa_0$ is the initial permeability. The percolation threshold porosity~$\hat{\theta} = p_c \theta_0$, represents the minimal porosity needed to have connection via voids or channels from one end to the other, and is dependent on the topology of the network. For $\hat{\theta} = 0.5 \theta_0$, the normalised permeabilities ($\kappa/\kappa_0$) for the network-inspired relation and the Kozeny-Carman relation are depicted in Fig.~\ref{Fig:Examples_relations} as function of the normalised porosity ($\theta/\theta_0$). In the coming section, we will show how the permeability-porosity relation~\eqref{Eq:NI} is derived using a network model.

\begin{figure}[!ht]
  \centering
  \includegraphics[scale=0.2]{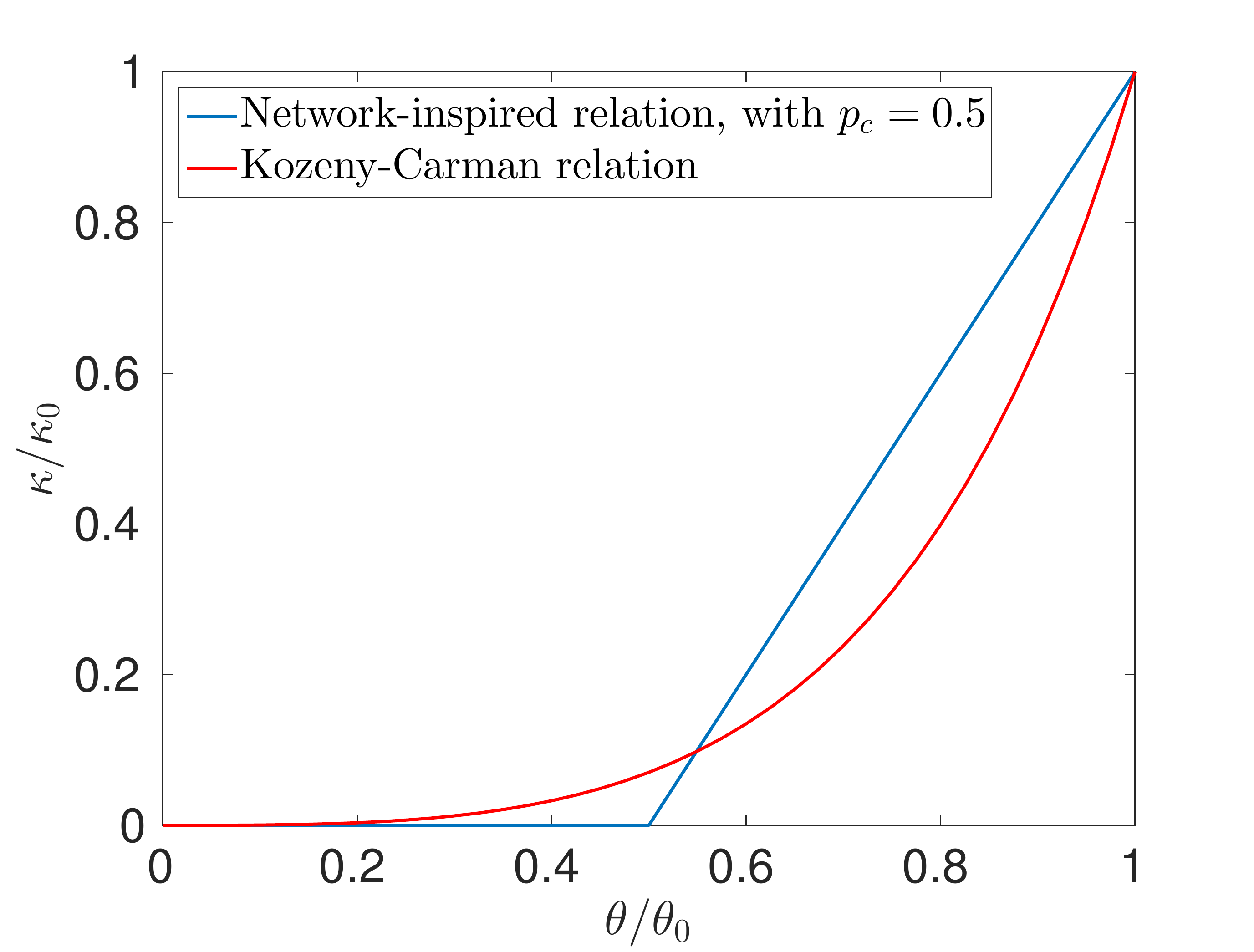}
  \caption{The normalised permeability as function of the normalised porosity for the Kozeny-Carman relation and the network-inspired relation, with $\hat{\theta} = 0.5 \theta_0$.} \label{Fig:Examples_relations}
\end{figure}

\subsection{The network-inspired permeability-porosity relation}

In this section, we explain the procedure followed to obtain the network-inspired permeability-porosity relation. According to Balberg~\cite{BALB87} and Wong~\cite{WONG88}, it is possible to obtain a relation between the number of open bonds (or channels) in the network $N_o$ and the flow rate through the network $Q$. This relation is determined by a power law as follows:
\begin{equation}
  Q \propto (N_o-\hat{N})^{b} \quad \mbox{for} \quad N_o>\hat{N}, \label{Eq:flow_percolation}
\end{equation}
in which $\hat{N}$ is the number of bonds corresponding with the percolation threshold porosity and $b$ is an exponent that can be determined by theory or via computer simulations. A similar relation between the permeability $\kappa$ and the porosity~$\theta$ will be obtained in this section.

First, we consider a network with all channels open and with a pressure gradient $\Delta p$ imposed over the horizontal direction of the network. Second, we assume that mass conservation holds in each node of the network $n_i$, hence
\begin{equation}
  \sum_{j \in S_i} q_{ij} = 0, \label{Eq:mass_conservation}
\end{equation}
where $S_i = \{ j\ | \mbox{\ node\ } n_{j} \mbox{ is adjacent to node } n_i\}$, and $q_{ij}$ is the flow rate in the channels connected to node $n_i$. This flow rate is given by the Poiseuille flow
\begin{equation}
  q_{ij} = \frac{\pi r^4}{8\eta l} \Delta p_{ij}, \label{Eq:Poiseuille}
\end{equation}
where $r$ and $l$ are the radius and the length of a channel between two neighbouring nodes, respectively, and $\Delta p_{ij}$ is the pressure drop.

A linear system for the pressure $p_i$ in each node arises from substituting Eq.~\eqref{Eq:Poiseuille} in Eq.~\eqref{Eq:mass_conservation}. This system is solved via a direct method and, subsequently, the flow rate in each channel is computed via Eq.~\eqref{Eq:Poiseuille}. Finally, the flow rate through the network $Q$ is computed by the summation of the flow rates in the channels that are connected with the outlet of the network. After this, we randomly close the channels, starting with $1\%$ of the total number of channels in the network until that $100\%$ of the channels is closed. In each stage, 500 simulations are performed and in each simulation, the linear system for the pressure is solved and the flow rate through the network is computed. From the computed flow rate $Q$, the permeability of the network can be determined using Darcy's law
\begin{equation}
  \kappa = -\frac{Q}{A} \frac{\eta L}{\Delta p},
\end{equation}
where $A$ is the cross-sectional area of the network and $L$ is the length over which the pressure gradient is taking place. In addition, there is a direct relation between the porosity of the network and the volume of the open channels $V_o$
\begin{equation}
  \theta = \frac{V_o}{V_t}\theta_0,
\end{equation}
in which $V_t$ is the total volume of the channels. This procedure yields a relation between the permeability and the porosity of the network. In the coming sections, this procedure will be demonstrated for three two-dimensional networks: rectangular, triangular and triangular unstructured.

\subsubsection{Rectangular network}

We start by describing the results obtained for a rectangular network (see Fig.~\ref{Fig:Rectangular_network}). In this case, we use a network with $N_x = 100$ horizontal nodes and $N_y = 60$ vertical nodes. We computed the fraction of closed channels $f_c = V_c/V_t$, that is needed to obtain a normalised permeability $\kappa_n = \kappa/\kappa_0$ such that $k_i - 0.05 < \kappa_n < k_i + 0.05$, where $k_i \in \{0.1, 0.2, \ldots, 0.9\}$. In Fig.~\ref{Fig:Rectangular}, the histograms for $k_i = 0.1$, $k_i = 0.5$ and $k_i = 0.9$ are shown. The mean $\bar{\mu}$ and the standard deviation $\sigma$ of the distribution of $f_c$ depend on the normalised permeability $\kappa_n$ as presented in Table~\ref{table6.1}.

\begin{figure}[!ht]
  \centering
  \subfloat[The histogram of $f_c$ that satisfies $0.05 < \kappa_n < 0.15$.]{\includegraphics[width=0.4\textwidth]{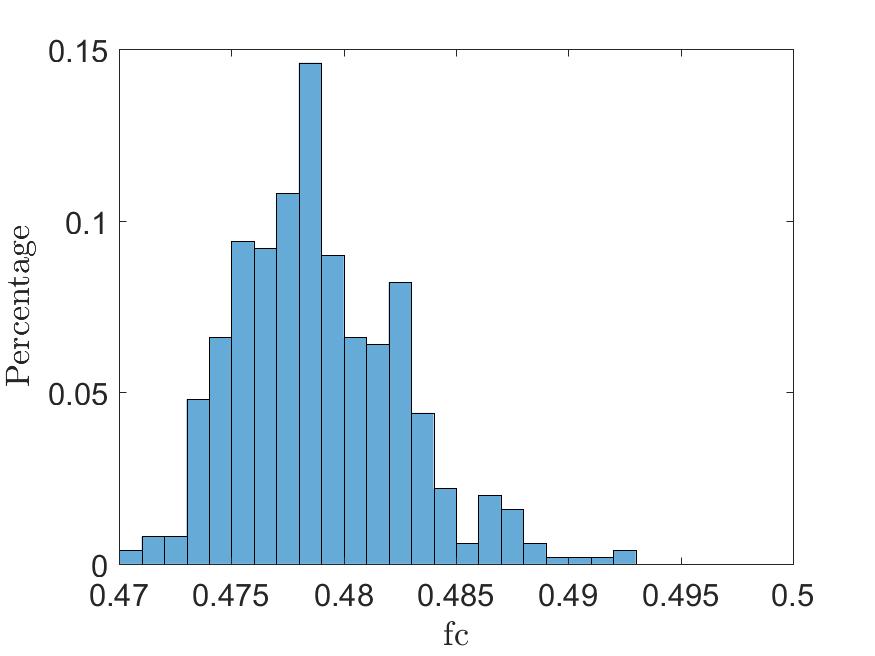}}%
  \qquad
  \subfloat[The histogram of $f_c$ that satisfies $0.45 < \kappa_n < 0.55$.]{\includegraphics[width=0.4\textwidth]{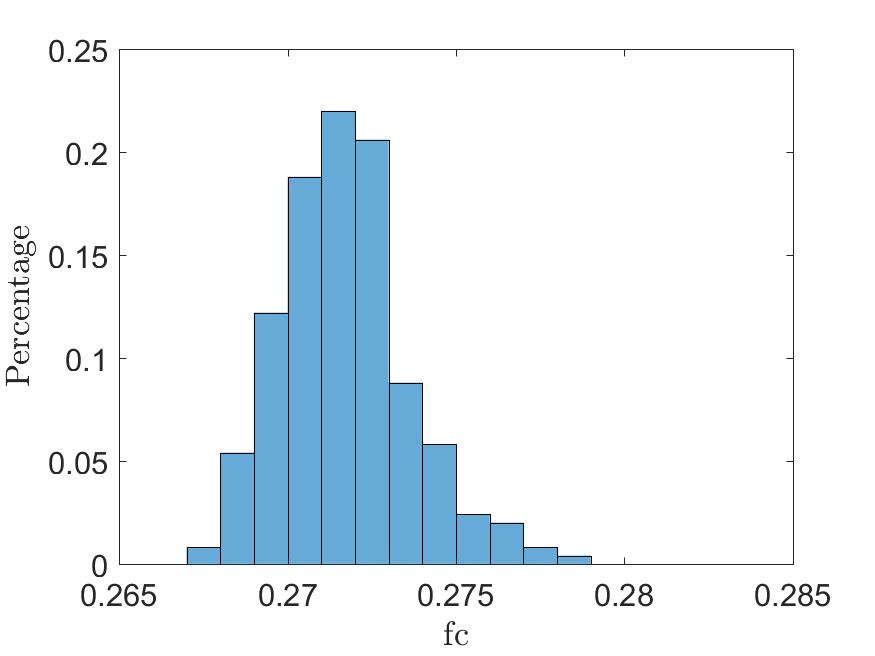}}%
  \\
  \subfloat[The histogram of $f_c$ that satisfies $0.85 < \kappa_n < 0.95$.]{\includegraphics[width=0.4\textwidth]{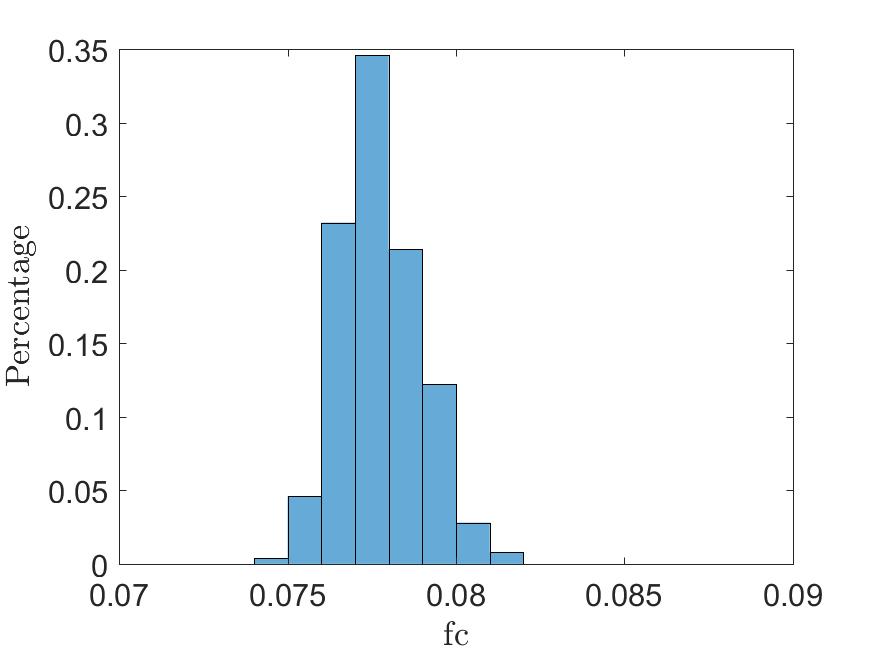}}
  \caption{The histograms of the fraction of closed channels $f_c$ for the rectangular network.} \label{Fig:Rectangular}
\end{figure}

\begin{table}[ht!]
  \centering
  \caption{The mean $\bar{\mu}$ and the standard deviation $\sigma$ of the distribution of $f_c$ for the rectangular network.} \label{table6.1}
  \begin{tabular}{lll}
    \hline\noalign{\smallskip}
    		 				& $\bar{\mu}$ 	& $\sigma$  \\
    \noalign{\smallskip}\hline\noalign{\smallskip}
    $0.05 < \kappa_n < 0.15$ 	& $0.4789$ 	& $0.0037$  \\
    $0.15 < \kappa_n < 0.25$ 	& $0.4188$ 	& $0.0026$  \\
    $0.25 < \kappa_n < 0.35$ 	& $0.3678$ 	& $0.0023$  \\
    $0.35 < \kappa_n < 0.45$	& $0.3195$ 	& $0.0022$  \\
    $0.45 < \kappa_n < 0.55$ 	& $0.2717$ 	& $0.0019$  \\
    $0.55 < \kappa_n < 0.65$ 	& $0.2240$ 	& $0.0018$  \\
    $0.65 < \kappa_n < 0.75$ 	& $0.1760$ 	& $0.0016$  \\
    $0.75 < \kappa_n < 0.85$ 	& $0.1273$ 	& $0.0014$  \\
    $0.85 < \kappa_n < 0.95$ 	& $0.0777$ 	& $0.0012$  \\
    \noalign{\smallskip}\hline
  \end{tabular}
\end{table}

\subsubsection{Triangular network}

In this section, we present the results obtained with a triangular network as shown in Fig.~\ref{Fig:Triangular_network}. The number of nodes in the horizontal direction $N_x = 100$ and the number of nodes in the vertical direction $N_y = 60$. The coordination number of the interior nodes is eight or four (see Fig.~\ref{Fig:Triangular_network}). In Fig.~\ref{Fig:Triangular}, the histograms for $k_i = 0.1$, $k_i = 0.5$ and $k_i = 0.9$ are depicted. In addition, the values of the mean $\bar{\mu}$ and the standard deviation~$\sigma$ of the distribution of $f_c$ are presented in Table~\ref{table6.2}.

\begin{figure}[!ht]
    \centering
    \includegraphics[width=0.27\textwidth]{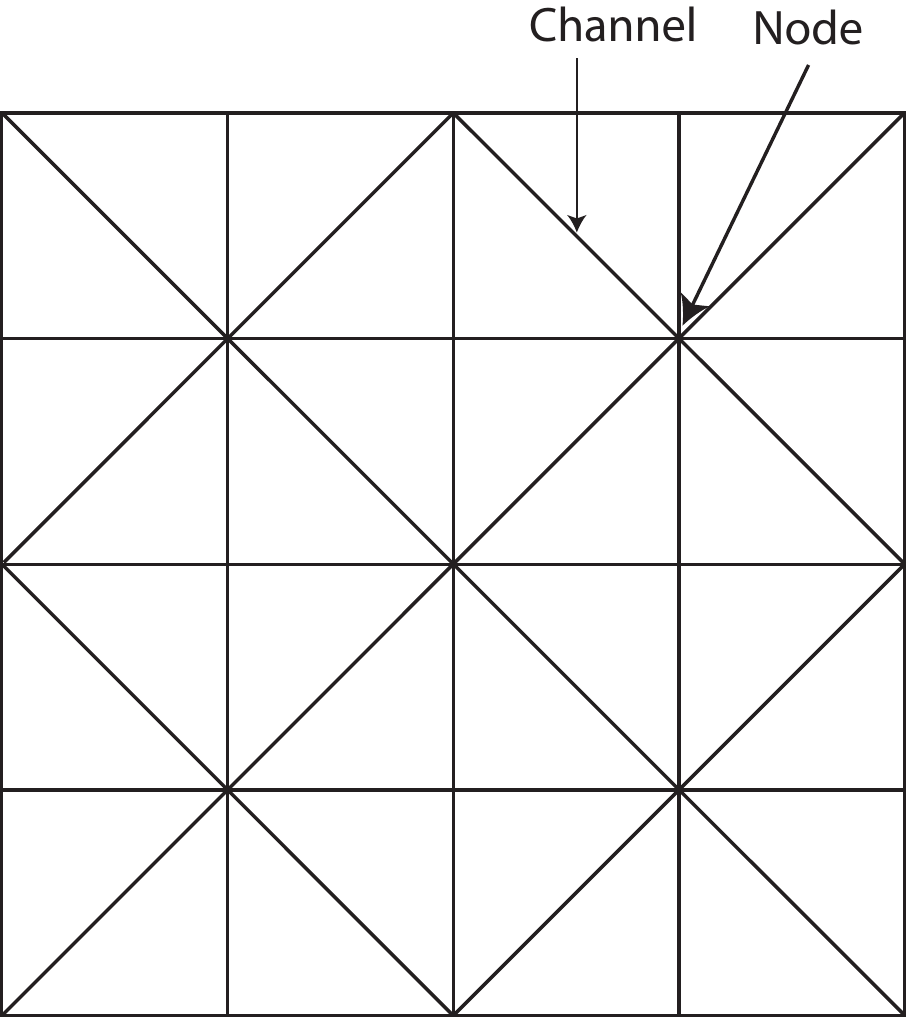}
    \caption{An illustration of a triangular network.} \label{Fig:Triangular_network}
\end{figure}

\begin{figure}[!ht]
  \centering
  \subfloat[The histogram of $f_c$ that satisfies $0.05 < \kappa_n < 0.15$.]{\includegraphics[width=0.4\textwidth]{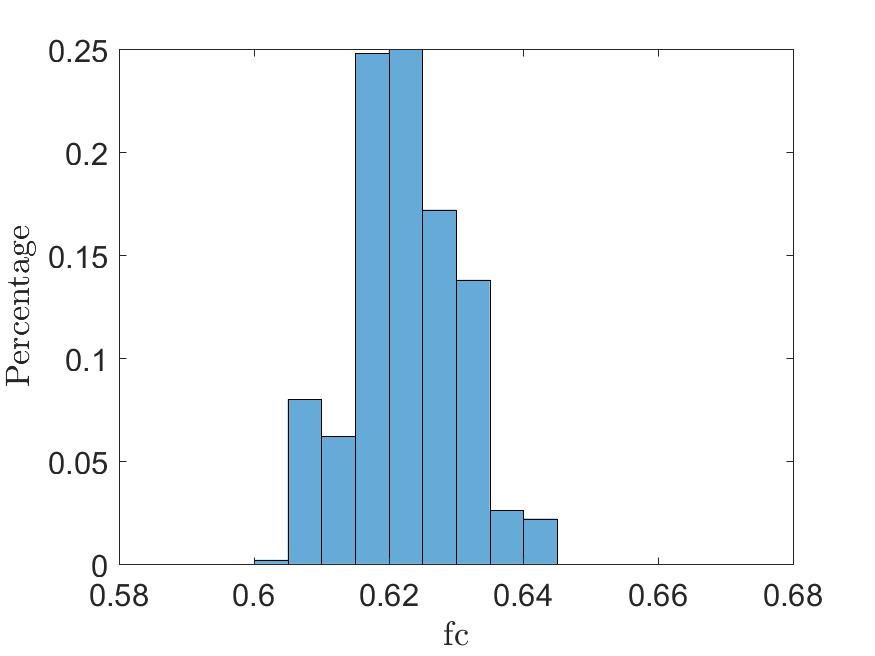}}%
  \qquad
  \subfloat[The histogram of $f_c$ that satisfies $0.45 < \kappa_n < 0.55$.]{\includegraphics[width=0.4\textwidth]{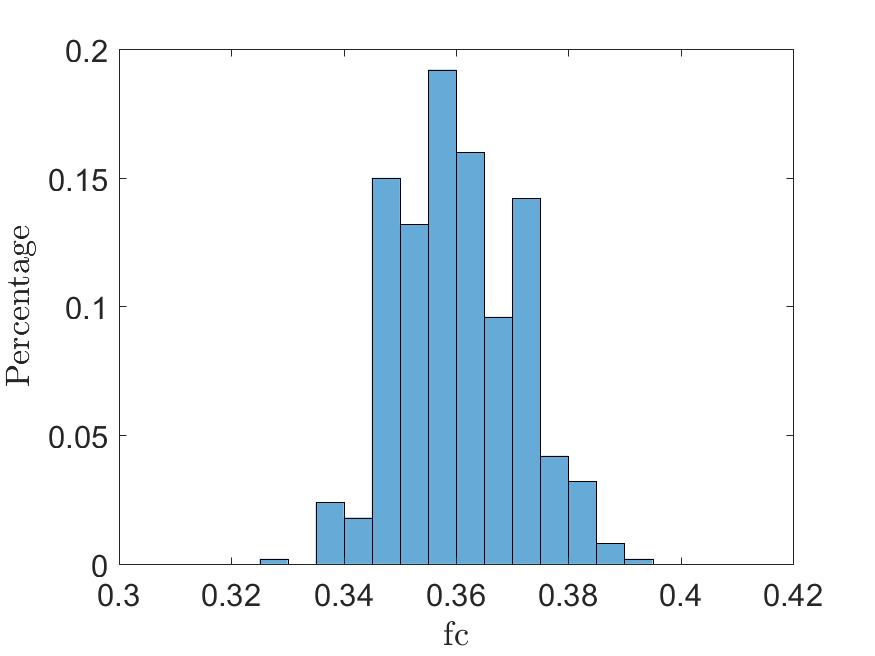}}%
  \\
  \subfloat[The histogram of $f_c$ that satisfies $0.85 < \kappa_n < 0.95$.]{\includegraphics[width=0.4\textwidth]{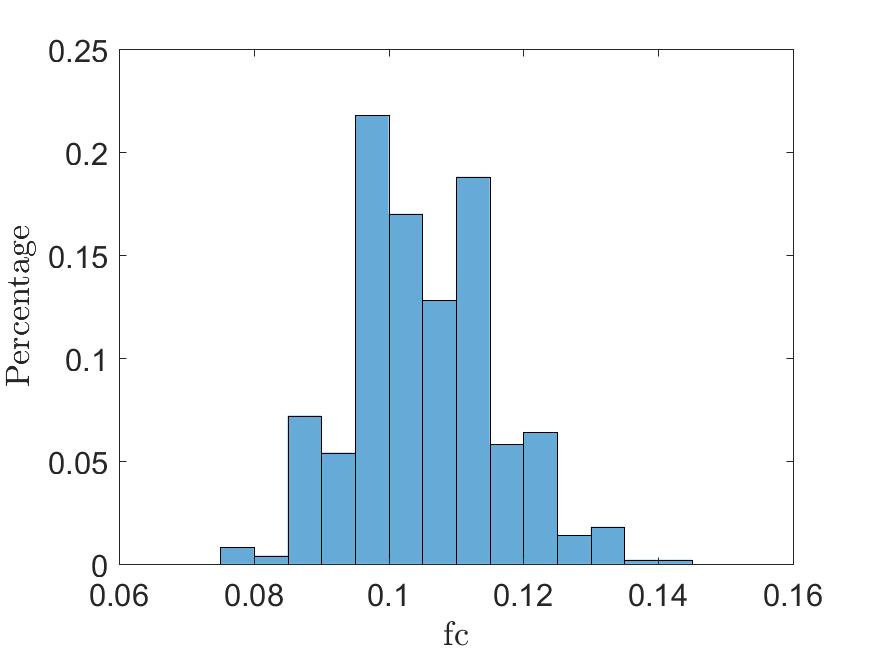}}
  \caption{The histograms of the fraction of closed channels $f_c$ for the triangular network.} \label{Fig:Triangular}
\end{figure}

\begin{table}[ht!]
  \centering
  \caption{The mean $\bar{\mu}$ and the standard deviation $\sigma$ of the distribution of $f_c$ for the triangular network.} \label{table6.2}
  \begin{tabular}{lll}
    \hline\noalign{\smallskip}
    		 				& $\bar{\mu}$ 	& $\sigma$  \\
    \noalign{\smallskip}\hline\noalign{\smallskip}
    $0.05 < \kappa_n < 0.15$ 	& $0.6226$ 	& $0.0077$  \\
    $0.15 < \kappa_n < 0.25$ 	& $0.5522$ 	& $0.0086$  \\
    $0.25 < \kappa_n < 0.35$ 	& $0.4882$ 	& $0.0099$  \\
    $0.35 < \kappa_n < 0.45$ 	& $0.4240$ 	& $0.0102$  \\
    $0.45 < \kappa_n < 0.55$ 	& $0.3604$ 	& $0.0105$  \\
    $0.55 < \kappa_n < 0.65$ 	& $0.2975$ 	& $0.0108$  \\
    $0.65 < \kappa_n < 0.75$ 	& $0.2333$ 	& $0.0111$  \\
    $0.75 < \kappa_n < 0.85$ 	& $0.1687$ 	& $0.0104$  \\
    $0.85 < \kappa_n < 0.95$ 	& $0.1052$ 	& $0.0103$  \\
    \noalign{\smallskip}\hline
  \end{tabular}
\end{table}

\subsubsection{Triangular unstructured network}

Finally, we present the results obtained with a triangular unstructured network such as shown in Fig.~\ref{fig_6.20}. The total number of nodes in this network is $6921$. The histograms of the fraction of closed channels for some values of the normalised permeability are depicted in Fig.~\ref{Fig:Unstructured_triangular}. Moreover, in Table~\ref{table6.3}, the values of the mean $\bar{\mu}$ and the standard deviation~$\sigma$ of the distribution of $f_c$ are given.

\begin{figure}
    \centering
    \includegraphics[width=0.3\textwidth]{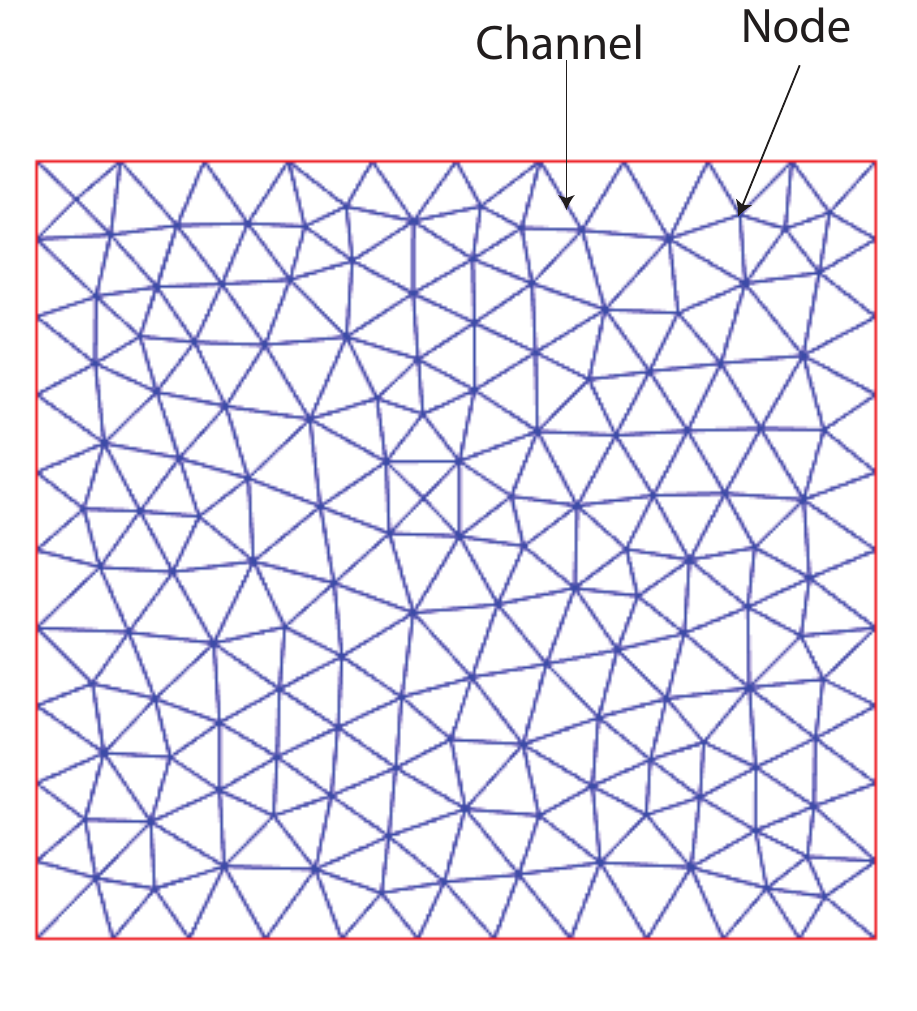}
    \caption{An illustration of a triangular unstructured network.} \label{fig_6.20}
\end{figure}

\begin{figure}[!ht]
  \centering
  \subfloat[The histogram of $f_c$ that satisfies $0.05 < \kappa_n < 0.15$.]{\includegraphics[width=0.4\textwidth]{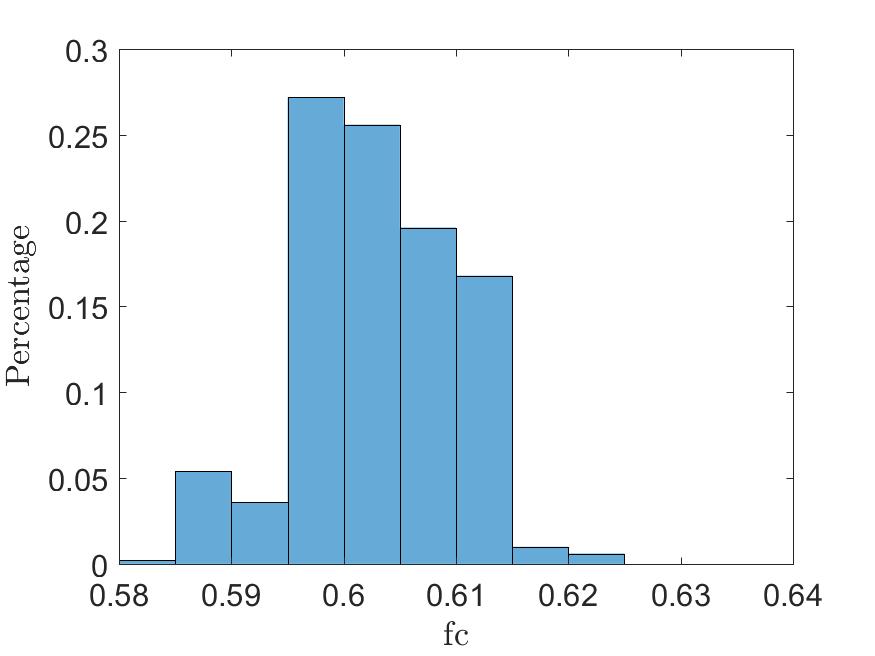}}%
  \qquad
  \subfloat[The histogram of $f_c$ that satisfies $0.45 < \kappa_n < 0.55$.]{\includegraphics[width=0.4\textwidth]{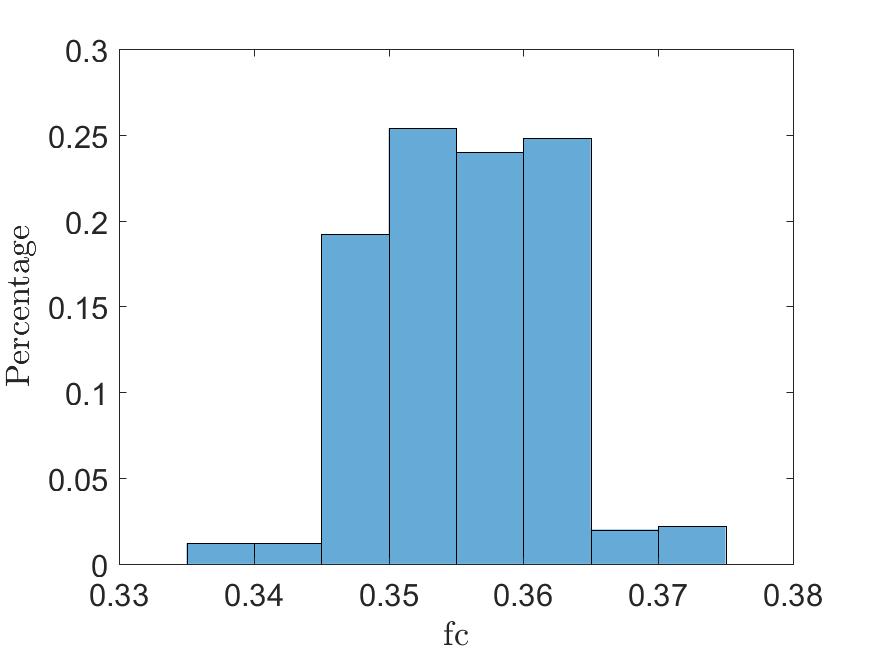}}%
  \\
  \subfloat[The histogram of $f_c$ that satisfies $0.85 < \kappa_n < 0.95$.]{\includegraphics[width=0.4\textwidth]{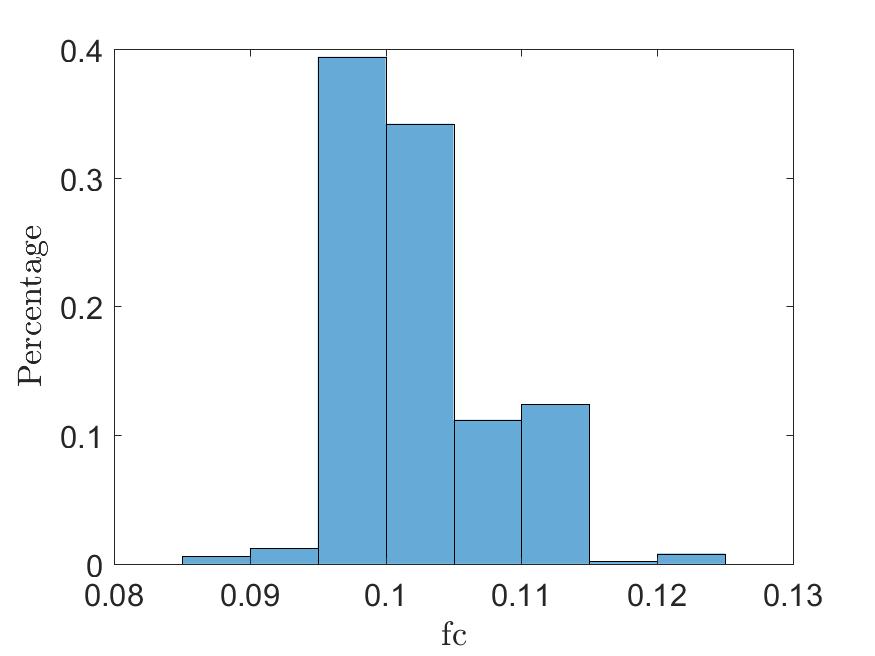}}
  \caption{The histograms of the fraction of closed channels $f_c$ for the triangular unstructured network.} \label{Fig:Unstructured_triangular}
\end{figure}

\begin{table}[ht!]
  \centering
  \caption{The mean $\bar{\mu}$ and the standard deviation $\sigma$ of the distribution of $f_c$ for the triangular unstructured network.} \label{table6.3}
  \begin{tabular}{lll}
    \hline\noalign{\smallskip}
    		 				& $\bar{\mu}$ 	& $\sigma$  \\
    \noalign{\smallskip}\hline\noalign{\smallskip}
    $0.05 < \kappa_n < 0.15$ 	& $0.6030$ 	& $0.0066$  \\
    $0.15 < \kappa_n < 0.25$ 	& $0.5385$ 	& $0.0062$  \\
    $0.25 < \kappa_n < 0.35$ 	& $0.4773$ 	& $0.0059$  \\
    $0.35 < \kappa_n < 0.45$ 	& $0.4168$ 	& $0.0063$  \\
    $0.45 < \kappa_n < 0.55$ 	& $0.3555$ 	& $0.0062$  \\
    $0.55 < \kappa_n < 0.65$ 	& $0.2938$ 	& $0.0063$  \\
    $0.65 < \kappa_n < 0.75$ 	& $0.2308$ 	& $0.0061$  \\
    $0.75 < \kappa_n < 0.85$ 	& $0.1667$ 	& $0.0059$  \\
    $0.85 < \kappa_n < 0.95$ 	& $0.1024$ 	& $0.0049$  \\
    \noalign{\smallskip}\hline
  \end{tabular}
\end{table}

This method is also used to compute $f_c$ for $\kappa_n = 0$. The mean $\bar{\mu}$ of the smallest value of $f_c$ for which holds $\kappa_n = 0$ is depicted in Table~\ref{table6.4}. Since $\theta/\theta_0 = 1 - f_c$, we observe from the results obtained with the different networks that the permeability becomes zero for porosities below a certain threshold. These percolation thresholds are different for the different networks used in this study, as can be seen in Table~\ref{table6.4}. The values are in good agreement with the literature~\cite{STWE19}. We also observe that, for all three networks, the relation between the permeability and the porosity exhibits an almost linear increase for values of the porosity larger than the percolation threshold. Moreover, the slope of these linear lines depends on the percolation threshold and hence on the network topology. The normalised permeabilities, for the different network topologies used in this paper and for the Kozeny-Carman equation, are depicted in Fig.~\ref{fig_6.40}.

\begin{table}[ht!]
  \centering
  \caption{The percolation threshold for different network topologies.} \label{table6.4}
  \begin{tabular}{lll}
    \hline\noalign{\smallskip}
    Network type			& $f_c$		& $p_c$ \\
    \noalign{\smallskip}\hline\noalign{\smallskip}
    Rectangular			& $0.5065$	& $0.4935$  \\
    Triangular 			& $0.6768$	& $0.3232$  \\
    Triangular unstructured	& $0.6562$	& $0.3438$  \\
    \noalign{\smallskip}\hline
  \end{tabular}
\end{table}

\begin{figure}[!ht]
  \centering
  \includegraphics[scale=0.2]{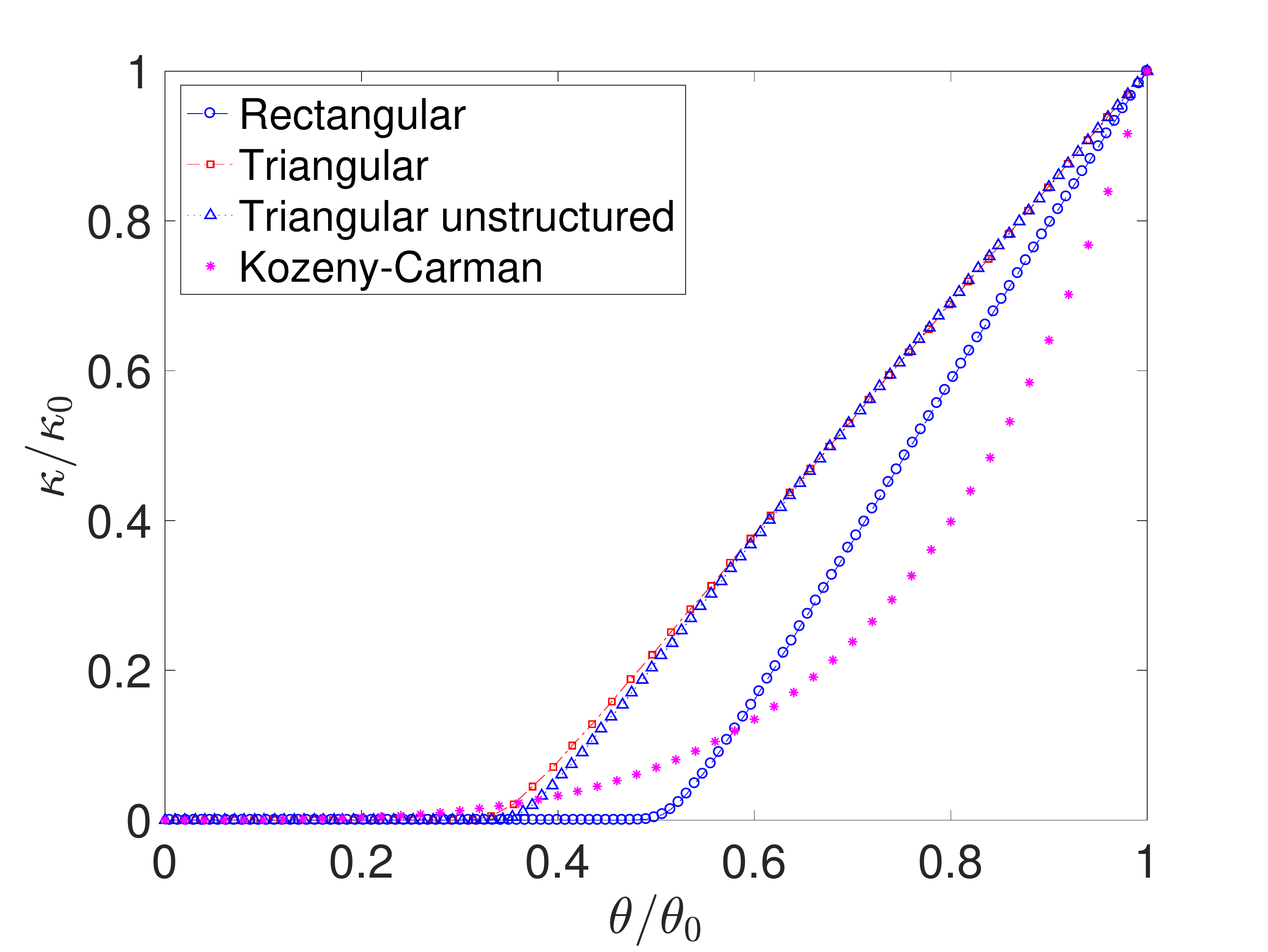}
  \caption{The network-inspired versus Kozeny-Carman based permeability-porosity relations.} \label{fig_6.40}
\end{figure}


\section{Problem formulation} \label{Sec:Problem_formulation}

The following numerical experiments are designed to study the different relations for the porosity and the permeability. In both experiments, the boundary conditions are chosen such that a decrease in porosity is realised in some parts of the computational domain.

\subsection{Problem with high pump pressure} \label{Sec:problem_formulation_highpressure}

In this numerical experiment, the infiltration of an incompressible fluid through a filter into a two-dimensional area is considered, as shown in Fig.~\ref{Fig:rectangle1}. During the infiltration, a high pump pressure is used to inject water into the porous medium. Leading to a compression of the material against the rigid right boundary $\Gamma_4$.

\begin{figure}[!ht]
  \centering
  \begin{tikzpicture}[thick]
    \draw (0, 0) -- (0, 3) -- (6, 3) -- (6, 0) -- (0, 0);
    \draw[color = dkgreen] (5.8, 2.8) node{$\Omega$};
    \draw[color = dkgreen] (3, 3.3) node{$\Gamma_1$};
    \draw[color = dkgreen] (-0.3, 1.5) node{$\Gamma_2$};
    \draw[color = dkgreen] (3, -0.3) node{$\Gamma_3$};
    \draw[color = dkgreen] (6.3, 1.5) node{$\Gamma_4$};
    \draw[<->] (0, -0.6) -- (6, -0.6);
    \draw (3, -0.9) node{$L$};
    \draw[<->] (-0.6, 0) -- (-0.6, 3);
    \draw (-0.9, 1.5) node{$H$};
  \end{tikzpicture}
\caption{Sketch of the setup for the two-dimensional problem with high pump pressure.} \label{Fig:rectangle1}
\end{figure}
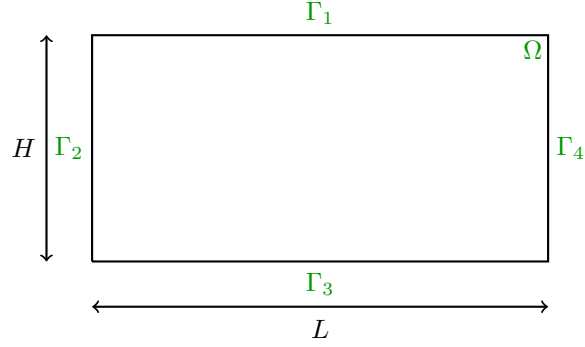

The computational domain $\Omega$ is a two-dimensional rectangular surface with Cartesian coordinates $\mathbf{x} = (x, y)$. In order to solve this problem, Biot's consolidation model, as described in Sect.~\ref{Sec:Governing_equations}, is applied on the computational domain~$\Omega$ with width $L$ and height $H$. The fluid is injected into the soil through a filter placed on boundary segment~$\Gamma_2$. More precisely, the boundary conditions for this problem are given as follows:
\begin{subequations}
  \begin{align}
    \frac{\kappa}{\eta} \nabla p \cdot \mathbf{n} = 0		&\mbox{\quad on\quad} \mathbf{x} \in \Gamma_1 \cup \Gamma_3; \\
    p = p_{pump}							&\mbox{\quad on\quad} \mathbf{x} \in \Gamma_2; \\
    p = 0									&\mbox{\quad on\quad} \mathbf{x} \in \Gamma_4; \\
    (\boldsymbol{\sigma}' \mathbf{n}) \cdot \mathbf{t} = 0	&\mbox{\quad on\quad} \mathbf{x} \in \Gamma_1 \cup \Gamma_3 \cup \Gamma_4; \\
    \mathbf{u} \cdot \mathbf{n} = 0                     			&\mbox{\quad on\quad} \mathbf{x} \in \Gamma_1 \cup \Gamma_3 \cup \Gamma_4; \label{Eq:BC1_5} \\
    \boldsymbol{\sigma}' \mathbf{n} = \mathbf{0}			&\mbox{\quad on\quad} \mathbf{x} \in \Gamma_2,
  \end{align} \label{Eq:BC1}
\end{subequations}
where $\mathbf{t}$ is the unit tangent vector at the boundary, $\mathbf{n}$ the outward unit normal vector and $p_{pump}$ is a prescribed high pump pressure due to the injection of the fluid. Figure~\ref{Fig:rectangle1} shows the definition of the boundary segments. Initially, the following condition is imposed:
\begin{equation}
  \mathbf{u}(\mathbf{x}, 0) = 0 \quad \mbox{for} \quad \mathbf{x} \in \Omega. \label{Eq:IC1}
\end{equation}

\subsection{Squeeze problem}

The infiltration of a fluid through a filter into a rectangular two-dimensional area is shown in Fig.~\ref{Fig:rectangle2}. In this numerical experiment, the porous medium is squeezed by applying a vertical load on the middle of the top and bottom edges of the domain.

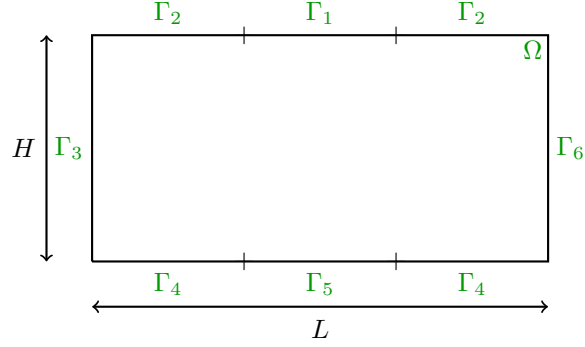
\begin{figure}[!ht]
  \centering
  \begin{tikzpicture}[thick]
    \draw (2, 0) node{+};
    \draw (2, 3) node{+};
    \draw (4, 0) node{+};
    \draw (4, 3) node{+};
    \draw (0, 0) -- (0, 3) -- (6, 3) -- (6, 0) -- (0, 0);
    \draw[color = dkgreen] (5.8, 2.8) node{$\Omega$};
    \draw[color = dkgreen] (3, 3.3) node{$\Gamma_1$};
    \draw[color = dkgreen] (1, 3.3) node{$\Gamma_2$};
    \draw[color = dkgreen] (5, 3.3) node{$\Gamma_2$};
    \draw[color = dkgreen] (-0.3, 1.5) node{$\Gamma_3$};
    \draw[color = dkgreen] (1, -0.3) node{$\Gamma_4$};
    \draw[color = dkgreen] (5, -0.3) node{$\Gamma_4$};
    \draw[color = dkgreen] (3, -0.3) node{$\Gamma_5$};
    \draw[color = dkgreen] (6.3, 1.5) node{$\Gamma_6$};
    \draw[<->] (0, -0.6) -- (6, -0.6);
    \draw (3, -0.9) node{$L$};
    \draw[<->] (-0.6, 0) -- (-0.6, 3);
    \draw (-0.9, 1.5) node{$H$};
  \end{tikzpicture}
\caption{Sketch of the setup for the two-dimensional squeeze problem.} \label{Fig:rectangle2}
\end{figure}

This problem is solved using Biot's consolidation model, that is applied on the computational domain $\Omega$ with width $L$ and height $H$. The fluid is injected into the soil through a filter placed on boundary segment~$\Gamma_3$. The boundary conditions for this problem are given as follows:
\begin{subequations}
  \begin{align}
    \frac{\kappa}{\eta} \nabla p \cdot \mathbf{n} = 0		&\mbox{\quad on\quad} \mathbf{x} \in \Gamma_1 \cup \Gamma_2 \cup \Gamma_4 \cup \Gamma_5; \\
    p = p_{pump}							&\mbox{\quad on\quad} \mathbf{x} \in \Gamma_3; \\
    p = 0									&\mbox{\quad on\quad} \mathbf{x} \in \Gamma_6; \\
    \boldsymbol{\sigma}' \mathbf{n} = (0, -\sigma'_0)^T	&\mbox{\quad on\quad} \mathbf{x} \in \Gamma_1; \label{Eq:BC2_4} \\
    \boldsymbol{\sigma}' \mathbf{n} = \mathbf{0}			&\mbox{\quad on\quad} \mathbf{x} \in \Gamma_2 \cup \Gamma_3 \cup \Gamma_4; \label{Eq:BC2_5} \\
    \boldsymbol{\sigma}' \mathbf{n} = (0, \sigma'_0)^T	&\mbox{\quad on\quad} \mathbf{x} \in \Gamma_5; \label{Eq:BC2_6} \\
    \mathbf{u} = \mathbf{0}							&\mbox{\quad on\quad} \mathbf{x} \in \Gamma_6,
  \end{align} \label{Eq:BC2}
\end{subequations}
where $\mathbf{t}$ is the unit tangent vector at the boundary, $\mathbf{n}$ the outward unit normal vector, $p_{pump}$ is a prescribed pump pressure due to the injection of the fluid and $\sigma'_0$ is the intensity of a uniform vertical load. Figure~\ref{Fig:rectangle2} shows the definition of the boundary segments. Initially, the following condition is imposed:
\begin{equation}
  \mathbf{u}(\mathbf{x}, 0) = 0 \quad \mbox{for} \quad \mathbf{x} \in \Omega. \label{Eq:IC2}
\end{equation}


\section{Numerical method}

To present the variational formulation of problem~\eqref{Eq:Biot}, we first introduce the appropriate function spaces. Let $L^2(\Omega)$ be the Hilbert space of square integrable scalar-valued functions, with inner product $(f, g) = \int_\Omega fg\ d\Omega$. And let $H^1(\Omega)$ denote the subspace of $L^2(\Omega)$ of functions with first derivatives in~$L^2(\Omega)$. We further introduce the function spaces
\begin{align*}
  \mathscr{W}(\Omega) 	&= \{\mathbf{w} \in (H^1(\Omega))^2 : (\mathbf{w} \cdot \mathbf{n})|_{\Gamma_1 \cup \Gamma_3 \cup \Gamma_4} = 0 \}; \\
  \mathscr{Q}(\Omega) 	&= \{q \in H^1(\Omega) : q|_{\Gamma_2} = p_{pump} \mbox{\ and\ } q|_{\Gamma_4} = 0 \}.
\end{align*}
In addition, we consider the bilinear forms
\begin{align*}
  a(\mathbf{u}, \mathbf{w}) 	&= \lambda(\nabla \cdot \mathbf{u}, \nabla \cdot \mathbf{w}) + 2 \mu \sum_{i, j = 1}^2 (\epsilon_{ij}(\mathbf{u}), \epsilon_{ij}(\mathbf{w})); \\
  b(p, \mathbf{w})       	&= (p, \nabla \cdot \mathbf{w}); \\
  c(p, q)             		&= \sum_{i = 1}^2 (\frac{\kappa}{\eta} \frac{\partial p}{\partial x_i}, \frac{\partial q}{\partial x_i}).
\end{align*}
The variational formulation for problem~\eqref{Eq:Biot} with boundary and initial conditions~\eqref{Eq:BC1}-\eqref{Eq:IC1}, consists of the following, using the notation $\dot{\mathbf{u}} = \frac{\partial \mathbf{u}}{\partial t}$: Find $(\mathbf{u}(t), p(t)) \in (\mathscr{W} \times \mathscr{Q})$ such that
\begin{subequations}
  \begin{align}
    a(\mathbf{u}(t), \mathbf{w}) - b(p(t), \mathbf{w}) 		&= h(\mathbf{w}) 	&\forall \mathbf{w} \in \mathscr{W}; \\
    b(q, \dot{\mathbf{u}}(t)) + c(p(t), q) 	&= 0 	&\forall q \in \mathscr{Q}_0,
  \end{align} \label{Eq:Variationalform}
\end{subequations}
with the initial condition $\mathbf{u}(0) = \mathbf{0},$ and where
\begin{align*}
  h(\mathbf{w})  			&= -p_{pump} \int_{\Gamma_2} \mathbf{w} \cdot \mathbf{n} \ d\Gamma; \\
  \mathscr{Q}_0(\Omega) 	&= \{q \in H^1(\Omega): q|_{\Gamma_2 \cup \Gamma_4} = 0 \}.
\end{align*}
Subsequently, problem~\eqref{Eq:Variationalform} is solved by applying the finite element method. Let $\mathscr{P}_h^k \subset H^1(\Omega)$ be a function space of piecewise polynomials on $\Omega$ of degree $k$. Hence, we define finite element approximations for $\mathscr{W}$ and $\mathscr{Q}$ as $\mathscr{W}_h^k = \mathscr{W} \cap (\mathscr{P}_h^k \times \mathscr{P}_h^k)$ with basis $\{\pmb{\phi}_i = (\phi_i, \phi_i) \in (\mathscr{W}_h^k \times \mathscr{W}_h^k): i = 1, \ldots, n_u \}$ and $\mathscr{Q}_h^{k'} = \mathscr{Q} \cap \mathscr{P}_h^{k'}$ with basis $\{\psi_j \in \mathscr{Q}_h^{k'}: j = 1, \ldots, n_p \}$, respectively~\cite{AGLR08}. Afterwards, we approximate the functions $\mathbf{u}(t)$ and $p(t)$ with functions $\mathbf{u}_h(t) \in \mathscr{W}_h^k$ and $p_h(t) \in \mathscr{Q}_h^{k'}$, defined as
\begin{equation}
  \mathbf{u}_h(t) = \sum_{i = 1}^{n_u} \mathbf{u}_i(t) \pmb{\phi}_i, \quad p_h(t) = \sum_{j = 1}^{n_p} p_j(t) \psi_j,
\end{equation}
in which the Dirichlet boundary conditions are imposed. Simultaneously, discretisation in time is applied using the backward Euler method. Let $\tau$ be the time step size and define a time grid $\{t_m = m \tau: m \in \mathbb{N}\}$, then the discrete Galerkin scheme of~\eqref{Eq:Variationalform} is formulated as follows: For $m \geq 1$, find $(\mathbf{u}_h^m, p_h^m) \in (\mathscr{W}_h^k \times \mathscr{Q}_h^{k'})$ such that
\begin{align}
  a(\mathbf{u}_h^m, \mathbf{w}_h) - b(p_h^m, \mathbf{w}_h) 	&= h(\mathbf{w}_h) \quad \quad\ \forall \mathbf{w}_h \in \mathscr{W}_{h}^k; \label{Eq:Discrete_Galerkin1} \\
  b(q_h, \mathbf{u}_h^m) + \tau c(p_h^m, q_h) 		&= b(q_h, \mathbf{u}_h^{m-1}) \ \forall q_h \in \mathscr{Q}_{0h}^{k'}, \label{Eq:Discrete_Galerkin2}
\end{align}
while for $m = 0$: $\mathbf{u}_h^0 = \mathbf{0}$. The discrete Galerkin scheme for problem~\eqref{Eq:Biot} with boundary and initial conditions~\eqref{Eq:BC2}-\eqref{Eq:IC2} is derived similarly.

In the network-inspired relation~\eqref{Eq:NI}, the permeability is equal to zero for $\theta \in [0, \hat{\theta}]$. Hence for large percolation thresholds, is it more likely that the permeability goes to zero. In the next section, the convergence of problem~\eqref{Eq:Biot} to a saddle point problem when $\kappa$ goes to zero is proven.

\subsection{Convergence to a saddle point problem} \label{Sec:Proof}

In this section, we will prove that the solution of the variational formulation~\eqref{Eq:Variationalform} converges to the solution of the related saddle point problem as $\tau \kappa \rightarrow 0$, both in the continuous as in the discrete system. Since we are interested in the case $\tau \kappa \rightarrow 0$, we can assume without loss of generality that $\kappa$ is bounded and that it fulfils the following relation (see Fig.~\ref{fig_6.40}):
\begin{equation}
  \kappa(\mathbf{x}, t) = \kappa_{max} f(\theta),
\end{equation}
where $\kappa_{max} > 0$ is constant and $f(\theta) \in [0, 1]$. Consequently, define the bilinear form
\begin{equation}
  c_\theta(p, q) = \sum_{i = 1}^2 (\frac{f(\theta)}{\eta} \frac{\partial p}{\partial x_i}, \frac{\partial q}{\partial x_i}),
\end{equation}
and define $\mathbf{\widetilde{u}} = \mathbf{u}^m - \mathbf{u}^{m-1}$, which gives $\mathbf{u}^m = \mathbf{\widetilde{u}} + \mathbf{u}^{m-1}$. This transformation results in the following variational formulation: For $m \geq 1$, find $(\mathbf{\widetilde{u}}, p^m) \in (\mathscr{W} \times \mathscr{Q})$ such that
\begin{equation*}
  (\mathcal{P}_{\tau \kappa}) :=
  \begin{cases}
    a(\mathbf{\widetilde{u}}, \mathbf{w}) - b(p^m, \mathbf{w}) = g(\mathbf{u}, \mathbf{w})	& \forall \mathbf{w} \in \mathscr{W}; \\
    b(q, \mathbf{\widetilde{u}}) + \tau \kappa_{max} c_\theta(p^m, q) = 0		& \forall q \in \mathscr{Q}_0,
  \end{cases}
\end{equation*}
where $g(\mathbf{u}, \mathbf{w}) = h(\mathbf{w}) - a(\mathbf{u}^{m-1}, \mathbf{w})$. The well-posedness of this problem has been analysed in~\cite{SHOW00,ZENI84}.

Using the inequality $\abs{2ab} \leq (a^2 + b^2)$ for $a,b \in \mathbb{R}$, we can show that
\begin{align}
  \abs{c_\theta(p^m,q)} &\leq \frac{1}{2\eta} \left[\sum_{i=1}^2 (\frac{\partial p^m}{\partial x_i},\frac{\partial p^m}{\partial x_i})+\sum_{i = 1}^2 (\frac{\partial q}{\partial x_i},\frac{\partial q}{\partial x_i})\right] \nonumber\\
			&\leq \frac{1}{2\eta} (\Vert p^m \Vert^2_{H^1(\Omega)} + \Vert q \Vert^2_{H^1(\Omega)}) < \infty, \label{Eq:c_bounded}
\end{align}
where the second inequality is a result of the definition of the $H^1$-norm
\begin{equation}
  \Vert p \Vert^2_{H^1(\Omega)} = (p, p) + \sum_{i = 1}^2 (\frac{\partial p}{\partial x_i}, \frac{\partial p}{\partial x_i}).
\end{equation}
Furthermore, define the spaces $\mathscr{R}(\Omega) = \{q \in L^2(\Omega) : q|_{\Gamma_2} = p_{pump} \mbox{\ and\ } q|_{\Gamma_4} = 0 \}$ and $\mathscr{R}_0(\Omega) = \{q \in L^2(\Omega) : q|_{\Gamma_2 \cup \Gamma_4} = 0 \}$. Hence, assuming that $\tau \kappa_{max} \rightarrow 0$, the variational formulation becomes: For $m \geq 1$, find $(\mathbf{\widetilde{u}}_0, p_0^m) \in (\mathscr{W} \times \mathscr{R})$ such that
\begin{equation*}
  (\mathcal{P}_0) :=
  \begin{cases}
    a(\mathbf{\widetilde{u}}_0, \mathbf{w}) - b(p_0^m, \mathbf{w}) = g(\mathbf{u}, \mathbf{w})	& \forall \mathbf{w} \in \mathscr{W}; \\
    b(q, \mathbf{\widetilde{u}}_0) = 0						& \forall q \in \mathscr{R}_0.
  \end{cases}
\end{equation*}
Brenner and Scott~\cite{BRSC07} showed, using Korn's inequality, that the bilinear form $a(\mathbf{u}, \mathbf{w})$ is coercive in $\mathscr{W}$. In addition, the bilinear form $b(p, \mathbf{w})$ satisfies the inf-sup condition on $(\mathscr{W} \times \mathscr{R})$ as proven in~\cite{BRAE07}. The continuity of both bilinear forms follows from the Cauchy-Schwarz inequality. Hence, from Brezzi's splitting theorem~\cite{BRAE07} we can conclude that a unique solution exists for problem $(\mathcal{P}_{0})$. This problem is referred to as a saddle point problem. We assume that our domain allows classical solutions to exist for problem $(\mathcal{P}_0)$, and as a result the solution of this problem is in the spaces $(\mathscr{W} \times \mathscr{Q})$. Then we prove that solutions to $(\mathcal{P}_{\tau \kappa})$ converge to solutions of $(\mathcal{P}_0)$, which illustrates that the solutions of $(\mathcal{P}_{\tau \kappa})$ inherit the saddle point structure of $(\mathcal{P}_0)$.

\begin{theorem}
  The solution of $(\mathcal{P}_{\tau \kappa})$ converges to the solution of $(\mathcal{P}_0)$ for $\tau \kappa_{max} \rightarrow 0$.
\end{theorem}

\begin{proof}
Subtracting $(\mathcal{P}_{0})$ from $(\mathcal{P}_{\tau \kappa})$ yields
\begin{subequations}
  \begin{align}
    a(\mathbf{\widetilde{u}} - \mathbf{\widetilde{u}}_0, \mathbf{w}) - b(p^m - p_0^m, \mathbf{w}) 		&= 0	&\forall \mathbf{w} \in \mathscr{W}; \label{Eq:Diff1} \\
    b(q, \mathbf{\widetilde{u}} - \mathbf{\widetilde{u}}_0) + \tau \kappa_{max} c_\theta(p^m, q) 	&= 0 	&\forall q \in \mathscr{Q}_0. \label{Eq:Diff2}
  \end{align}
\end{subequations}
Take $\mathbf{w} = \mathbf{\widetilde{u}} - \mathbf{\widetilde{u}}_0$ in~\eqref{Eq:Diff1}, hence we get
\begin{equation}
  a(\mathbf{\widetilde{u}} - \mathbf{\widetilde{u}}_0, \mathbf{\widetilde{u}} - \mathbf{\widetilde{u}}_0) - b(p^m - p_0^m, \mathbf{\widetilde{u}} - \mathbf{\widetilde{u}}_0) = 0. \label{Eq:a_b}
\end{equation}
Using~\eqref{Eq:Diff2} and the bilinearity of the form $c_\theta(p, q)$ gives
\begin{equation*}
  a(\mathbf{\widetilde{u}} - \mathbf{\widetilde{u}}_0, \mathbf{\widetilde{u}} - \mathbf{\widetilde{u}}_0) + \tau \kappa_{max} c_\theta(p^m, p^m) = \tau \kappa_{max} c_\theta(p^m, p_0^m).
\end{equation*}
From the coercivity of $a(\mathbf{u}, \mathbf{w})$ and the definition of $c_\theta(p, q)$, we can state that
\begin{subequations}
  \begin{align}
    0 &\leq a(\mathbf{\widetilde{u}} - \mathbf{\widetilde{u}}_0, \mathbf{\widetilde{u}} - \mathbf{\widetilde{u}}_0) + \tau \kappa_{max} c_\theta(p^m, p^m) \\
      &= \tau \kappa_{max} c_\theta(p^m, p_0^m) < \infty,
  \end{align} \label{Eq:a_c}
\end{subequations}
where the last inequality follows from~\eqref{Eq:c_bounded}. Let $\tau \kappa_{max} \rightarrow 0$, then follows that $a(\mathbf{\widetilde{u}} - \mathbf{\widetilde{u}}_0, \mathbf{\widetilde{u}} - \mathbf{\widetilde{u}}_0) \rightarrow 0$. Consequently, the coercivity of $a(\mathbf{u}, \mathbf{w})$ implies that $\mathbf{\widetilde{u}} \rightarrow \mathbf{\widetilde{u}}_0$ as $\tau \kappa_{max} \rightarrow 0$. Therefore, we have $a(\mathbf{\widetilde{u}} - \mathbf{\widetilde{u}}_0, \mathbf{w}) \rightarrow 0$ for all $\mathbf{w} \in \mathscr{W}$. Subsequently, this last result together with Eq.~\eqref{Eq:Diff1} lead to $b(p^m - p_0^m, \mathbf{w}) \rightarrow 0$. Hence, we can conclude that $p^m \rightarrow p_0^m$ as $\tau \kappa_{max} \rightarrow 0$.
\end{proof}

Numerically, the discrete Galerkin scheme~\eqref{Eq:Discrete_Galerkin1}-\eqref{Eq:Discrete_Galerkin2} can be expressed as a linear equations system
\begin{equation*}
  (\mathcal{P}^D_{\tau \kappa}) :=
  \begin{cases}
    A \mathbf{u}^m - B^T \mathbf{p}^m			&= \mathbf{h}, \\
    B \mathbf{u}^m + \tau \kappa_{max} C \mathbf{p}^m	&= B \mathbf{u}^{m-1},
  \end{cases} \label{Eq:Numerical_tau}
\end{equation*}
with $\mathbf{u} = (\mathbf{u}_1, \mathbf{u}_2, \ldots , \mathbf{u}_{n_u})$ and $\mathbf{p} = (p_1, p_2, \ldots , p_{n_p})$. The matrix $A \in \mathbb{R}^{2n_u \times 2n_u}$ is the symmetric positive definite elasticity matrix, $B \in \mathbb{R}^{n_p \times 2n_u}$ the divergence matrix and $C \in \mathbb{R}^{n_p \times n_p}$ is the diffusive matrix. The vector $\mathbf{h}$ is the right-hand side vector with components $\mathbf{h}_i = h(\pmb{\phi}_i), i = 1, \ldots, n_u$. Initially, $\mathbf{u}^0 = \mathbf{0}$. For $\tau \kappa_{max} \rightarrow 0$, we have
\begin{equation*}
  (\mathcal{P}^D_0) :=
  \begin{cases}
    A \mathbf{u}_0^m - B^T \mathbf{p}_0^m	&= \mathbf{h}, \\
    B \mathbf{u}_0^m			&= B \mathbf{u}_0^{m-1}.
  \end{cases}
\end{equation*}

\begin{theorem}
  The solution of $(\mathcal{P}^D_{\tau \kappa})$ converges to the solution of $(\mathcal{P}^D_0)$ as $\tau \kappa_{max} \rightarrow 0$.
\end{theorem}

\begin{proof}
The convergence proof in the discrete problem is similar to the continuous problem. Hence, it is sufficient to show that $a(\mathbf{u}, \mathbf{w})$ is coercive in $\mathscr{W}_h^k$. Since $\mathscr{W}_h^k \subset \mathscr{W}$, the coercivity property is automatically verified in the discrete case.
\end{proof}

Since Biot's poroelasticity problem converges to the related saddle point problem as $\tau \kappa \rightarrow 0$, the Taylor-Hood elements can be used to solve this problem numerically. These elements represent finite element pairs that satisfy the inf-sup condition for the saddle point problem~\cite{BBF13,ERGU13}. Although the inf-sup condition and the coercivity and boundedness of bilinear form $a(\mathbf{u}, \mathbf{w})$ warrant existence and uniqueness of the finite element solution, the inf-sup condition is not sufficient for reliable numerical solutions of Biot's problem~\eqref{Eq:Biot}. Since for low permeabilities and/or small time steps, the approximation to the pressure exhibits nonphysical oscillations due to loss of the M-matrix property, as shown in~\cite{AGLR08}. In order to improve the monotonicity properties of the finite element scheme and to obtain oscillation free approximations of the pressure, the stabilisation procedure outlined in~\cite{AGLR08,RGHZ16} is applied in this study. Therefore, a stabilisation term is added to the continuity equation in Biot's model with stabilisation parameter $\beta = \frac{\sqrt{\Delta x^2 + \Delta y^2}}{4(\lambda+2\mu)}$.


\section{Numerical results}

The Galerkin finite element method, with triangular Taylor-Hood elements~\cite{RAVE18,SEGA12}, is adopted to solve the discretised quasi-two-dimensional problem~\eqref{Eq:Biot}. The displacements are spatially approximated by quadratic basis functions, whereas a continuous piecewise linear approximation is used for the pressure field. From the system of equations~\eqref{Eq:Biot} and the boundary conditions~\eqref{Eq:BC1} it is obvious that this two-dimensional problem is symmetrical in the $y$-direction and that it can be reduced to a one-dimensional problem. Aguilar et al.~\cite{AGLR08} solved this one-dimensional problem analytically and showed that the finite element method with Taylor-Hood elements gives accurate numerical results. For the time integration, the backward Euler method is applied. The numerical investigations are carried out using the matrix-based software package MATLAB (version R2015a).

The computational domain is a rectangular surface with width $L = 2.0\, \mathrm{m}$ and height $H = 1.0\, \mathrm{m}$. The domain is discretised using a regular triangular grid, with $\Delta x = \Delta y = 0.02$. In addition, values for some model parameters have been chosen based on literature (see Table~\ref{Tab:MaterialValues}).

\begin{table}[H]
  \centering
  \caption{An overview of the values of the model parameters.} \label{Tab:MaterialValues}
  \begin{tabular}{llll}
    \hline\noalign{\smallskip}
    Property			& Symbol		& Value                 					& Unit \\
    \noalign{\smallskip}\hline\noalign{\smallskip}
    Young's modulus		& $E$        	& $35 \cdot 10^6$              			& $\mathrm{Pa}$ \\
    Poisson's ratio		& $\nu$      	& $0.3$                 					& - \\
    Fluid viscosity		& $\eta$		& $1.307 \cdot 10^{-3}$				& $\mathrm{Pa \cdot s}$ \\
    Initial porosity		& $\theta_0$ 	& $0.4$                 					& - \\
    Mean grain size		& $d_s$      	& $0.2 \cdot 10^{-3}$   				& $\mathrm{m}$ \\
    Pump pressure		& $p_{pump}$	& $50 \cdot 10^{5}$ / $5 \cdot 10^{5}$	& $\mathrm{Pa}$ \\
    Uniform load			& $\sigma'_0$	& $3 \cdot 10^{6}$					& $\mathrm{N/m^2}$ \\
    \noalign{\smallskip}\hline
  \end{tabular}
\end{table}

Furthermore, the Lam\'e coefficients $\lambda$ and $\mu$ are related to Young's modulus $E$ and Poisson's ratio $\nu$ by
\begin{equation}
  \lambda = \frac{\nu E}{(1+\nu)(1-2\nu)},\quad \mu = \frac{E}{2(1+\nu)}. \label{Eq:Lame}
\end{equation}
The impact of the permeability-porosity relation on the water flow is defined in this study as the impact on the time average of the volumetric flow rate $\overline{Q}_{out}$ at a distance~$L$ from the injection filter. The initial permeability $\kappa_0$ used in Eq.~\eqref{Eq:NI} is computed by the Kozeny-Carman relation~\eqref{Eq:KC}, in order to have a reliable comparison between the two relations. In the generations of the simulation results, the time step size is chosen to be $\tau = 0.5$.

\subsection{Numerical results for the problem with high pump pressure}

In order to obtain some insight into the impact of a high pump pressure on the water flow, we present an overview of the simulation results in Figs.~\ref{Fig:Results_p_KC}-\ref{Fig:Results_p_Q}. In these simulations, water is injected into the soil at a constant pump pressure of $50\, \mathrm{bar}$. The simulated fluid velocity, permeability and porosity profiles that have been obtained using the Kozeny-Carman relation are provided in Fig.~\ref{Fig:Results_p_KC}, while the simulated results that have been obtained using the network-inspired relation with $p_c = 0.3232$, corresponding with a triangular structured network, are provided in Fig.~\ref{Fig:Results_p_T}. In Fig.~\ref{Fig:Results_p_Q}, the simulated results that have been obtained using the network-inspired relation with $p_c = 0.4935$, corresponding with a rectangular network, are depicted.

\begin{figure*}[!ht]
  \centering
  \subfloat[Numerical solution for the fluid velocity.]{\includegraphics[width=0.3\textwidth]{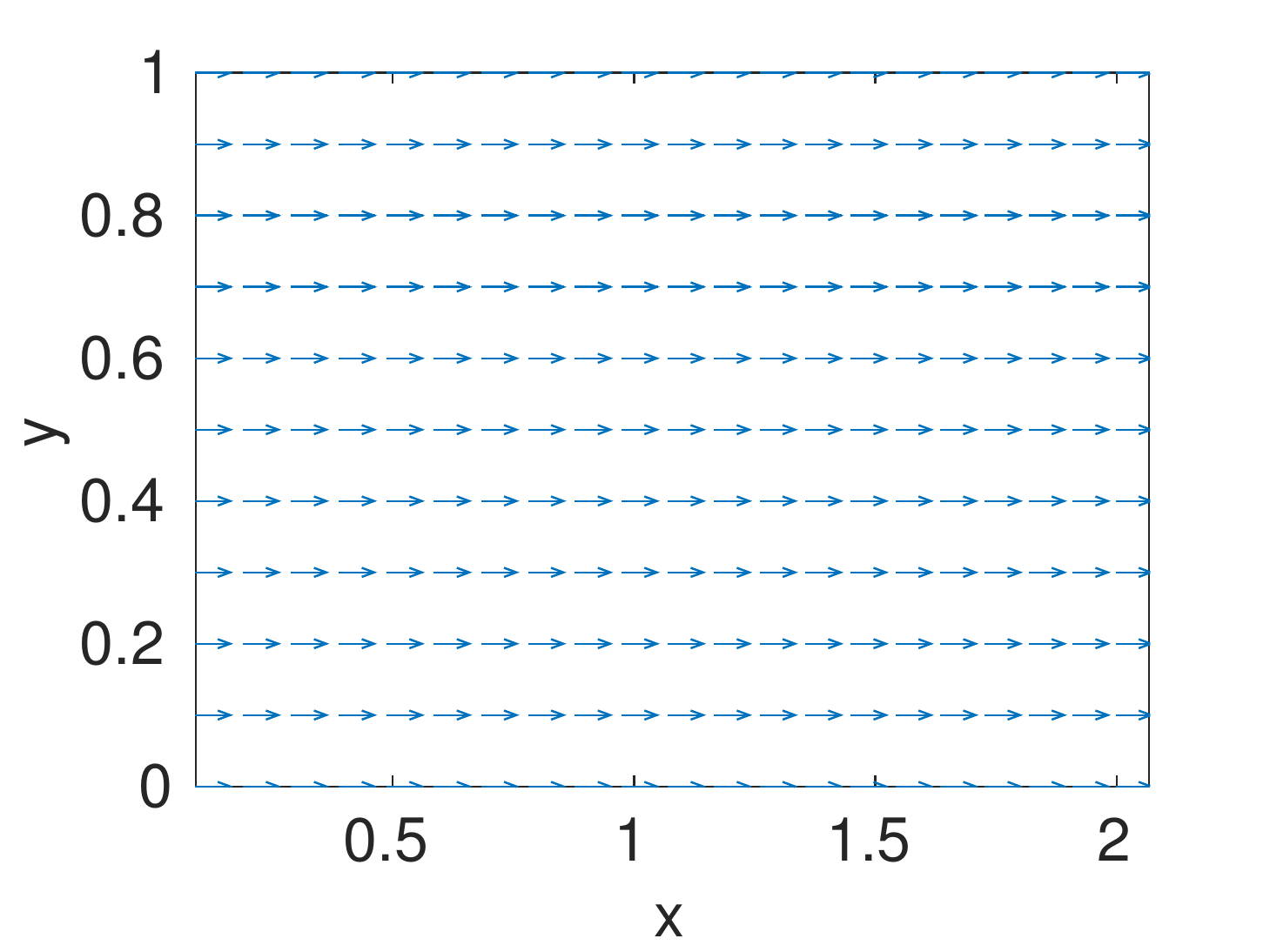}}%
  \qquad
  \subfloat[Numerical solution for the permeability. \label{Fig:Results_p_KC_kappa}]{\includegraphics[width=0.3\textwidth]{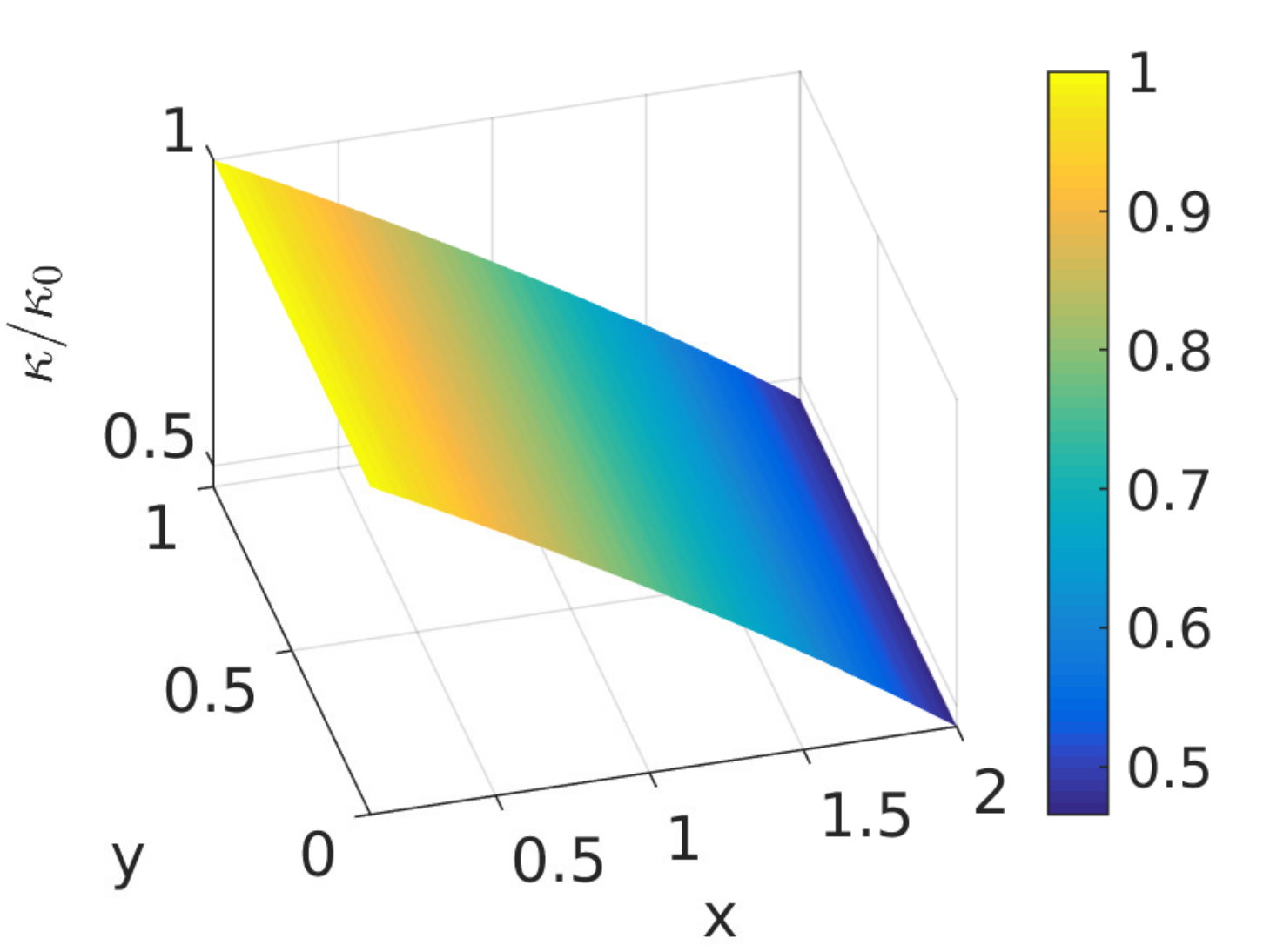}}%
  \qquad
  \subfloat[Numerical solution for the porosity. \label{Fig:Results_p_KC_theta}]{\includegraphics[width=0.3\textwidth]{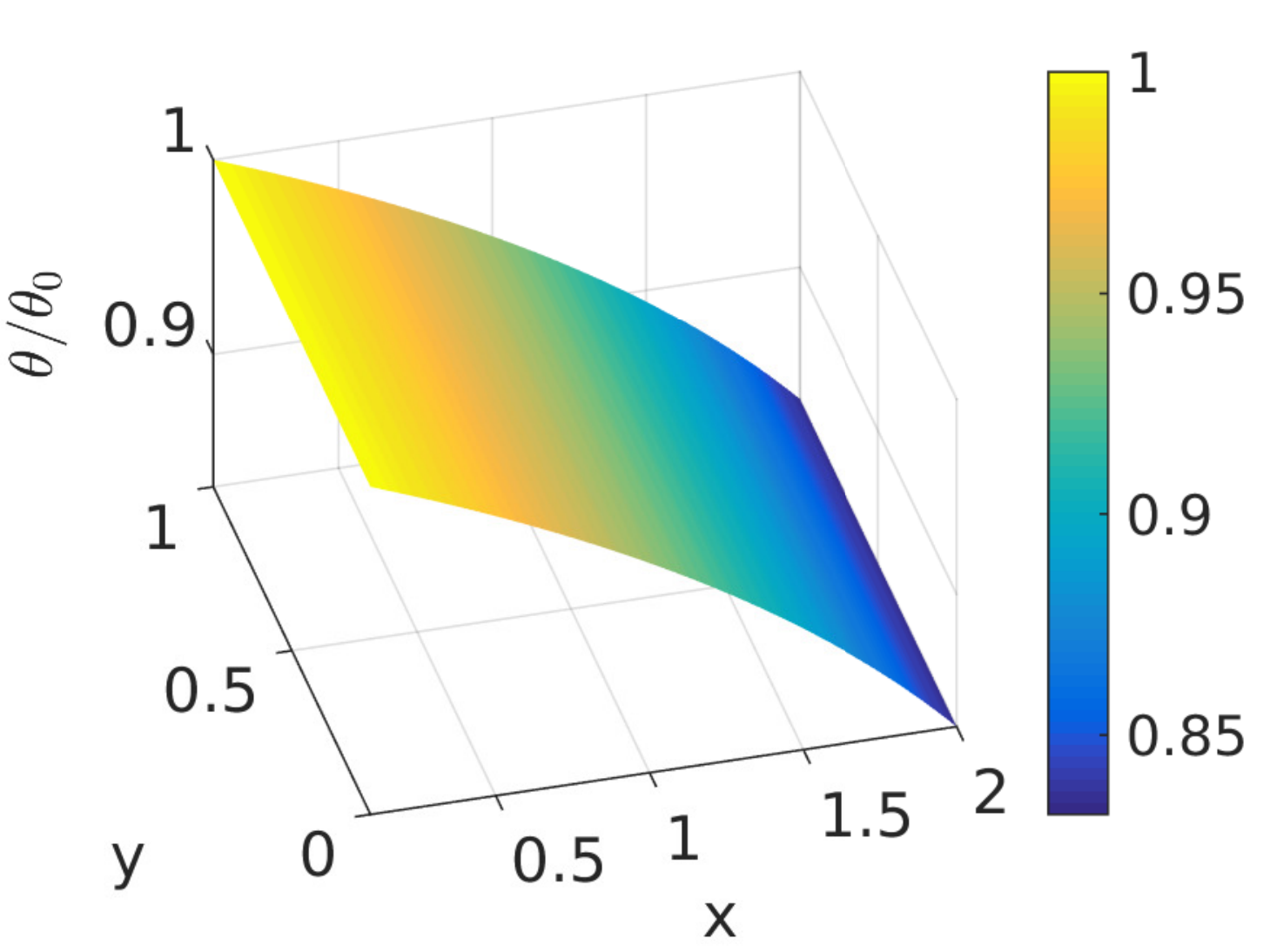}}
  
  \caption{Numerical solutions for the fluid velocity, the permeability and the porosity, at time $t = 300$, obtained using the Kozeny-Carman relation.} \label{Fig:Results_p_KC}
\end{figure*}

\begin{figure*}[!ht]
  \centering
  \subfloat[Numerical solution for the fluid velocity.]{\includegraphics[width=0.3\textwidth]{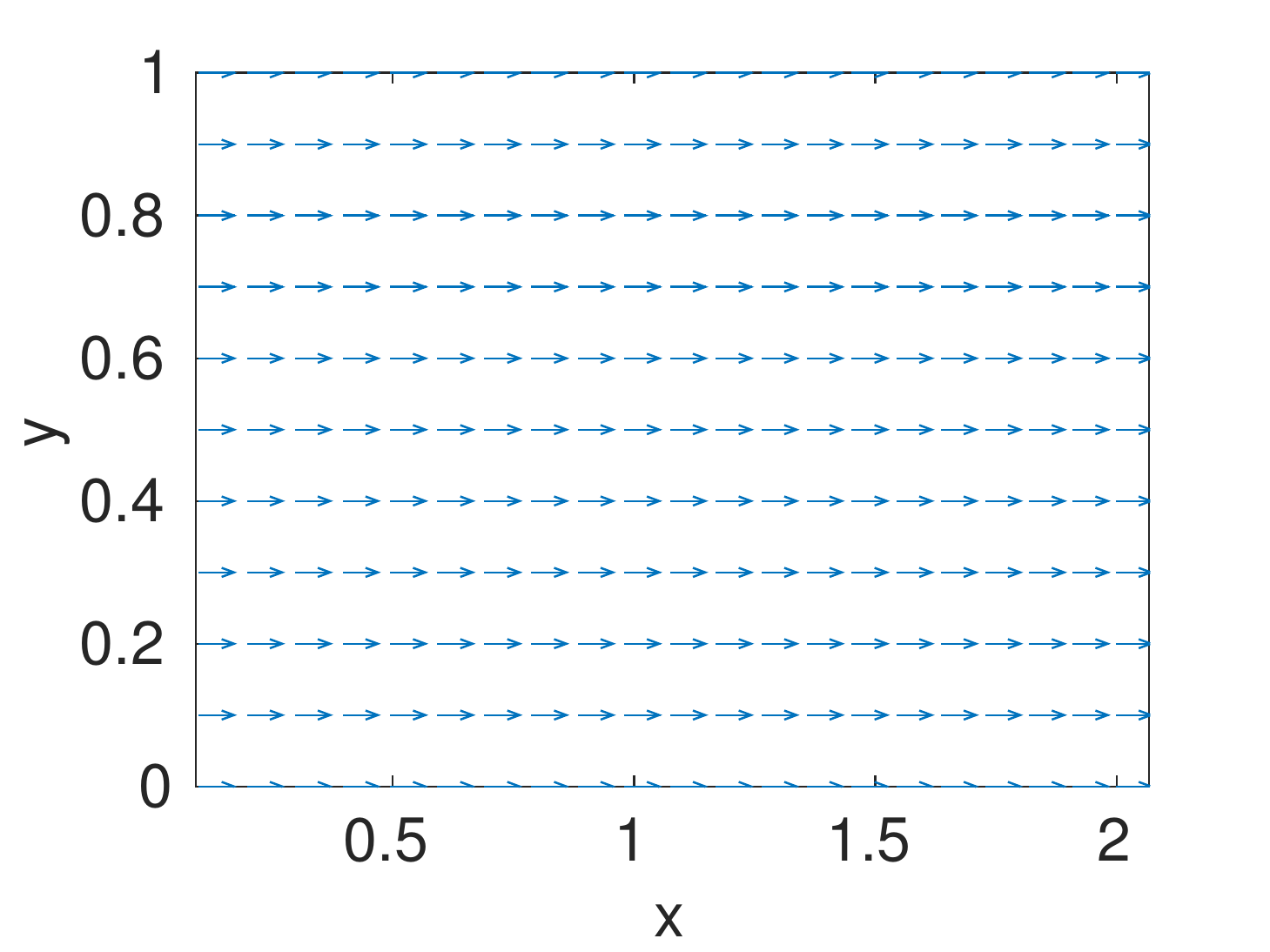}}%
  \qquad
  \subfloat[Numerical solution for the permeability. \label{Fig:Results_p_T_kappa}]{\includegraphics[width=0.3\textwidth]{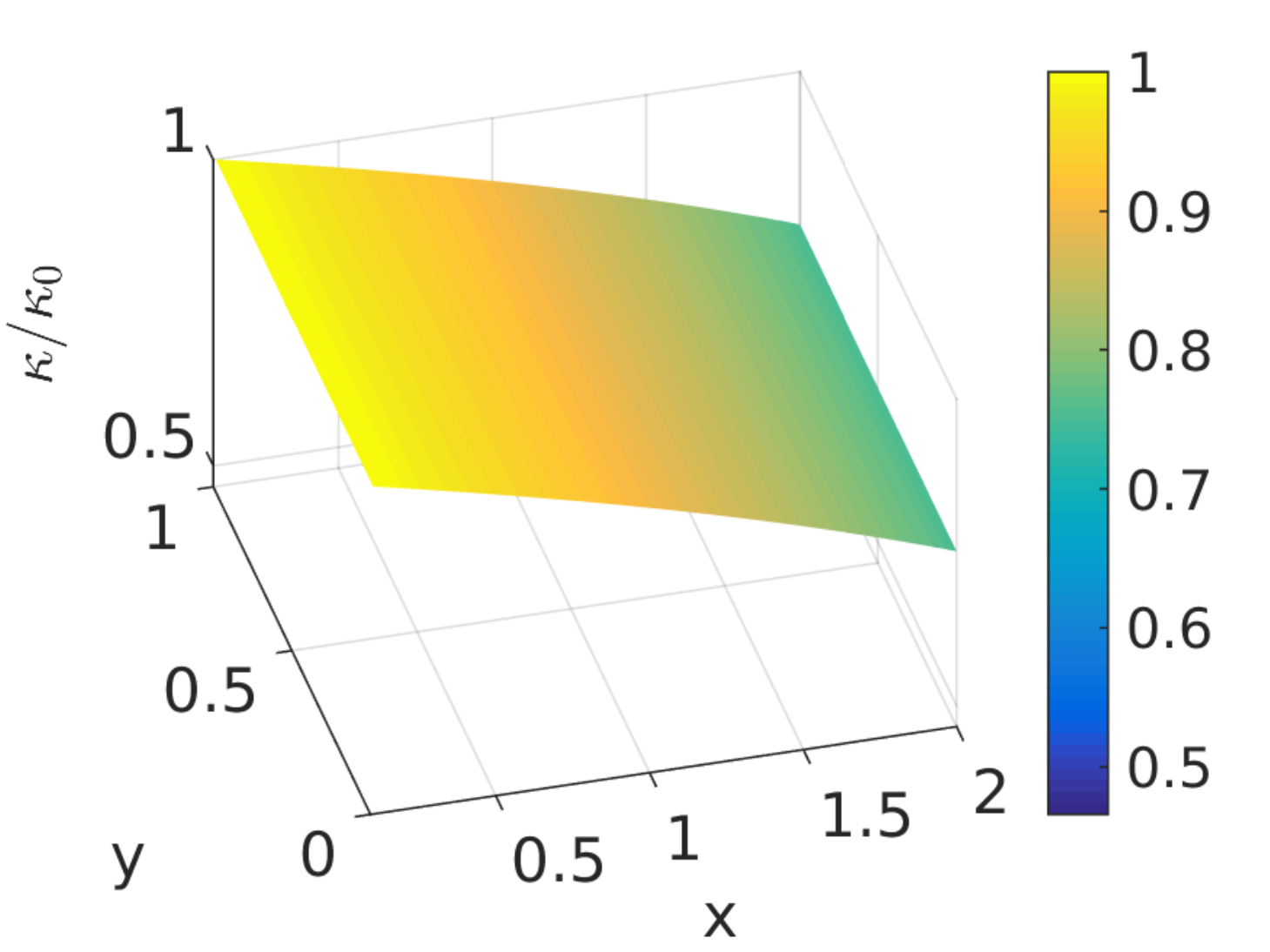}}%
  \qquad
  \subfloat[Numerical solution for the porosity. \label{Fig:Results_p_T_theta}]{\includegraphics[width=0.3\textwidth]{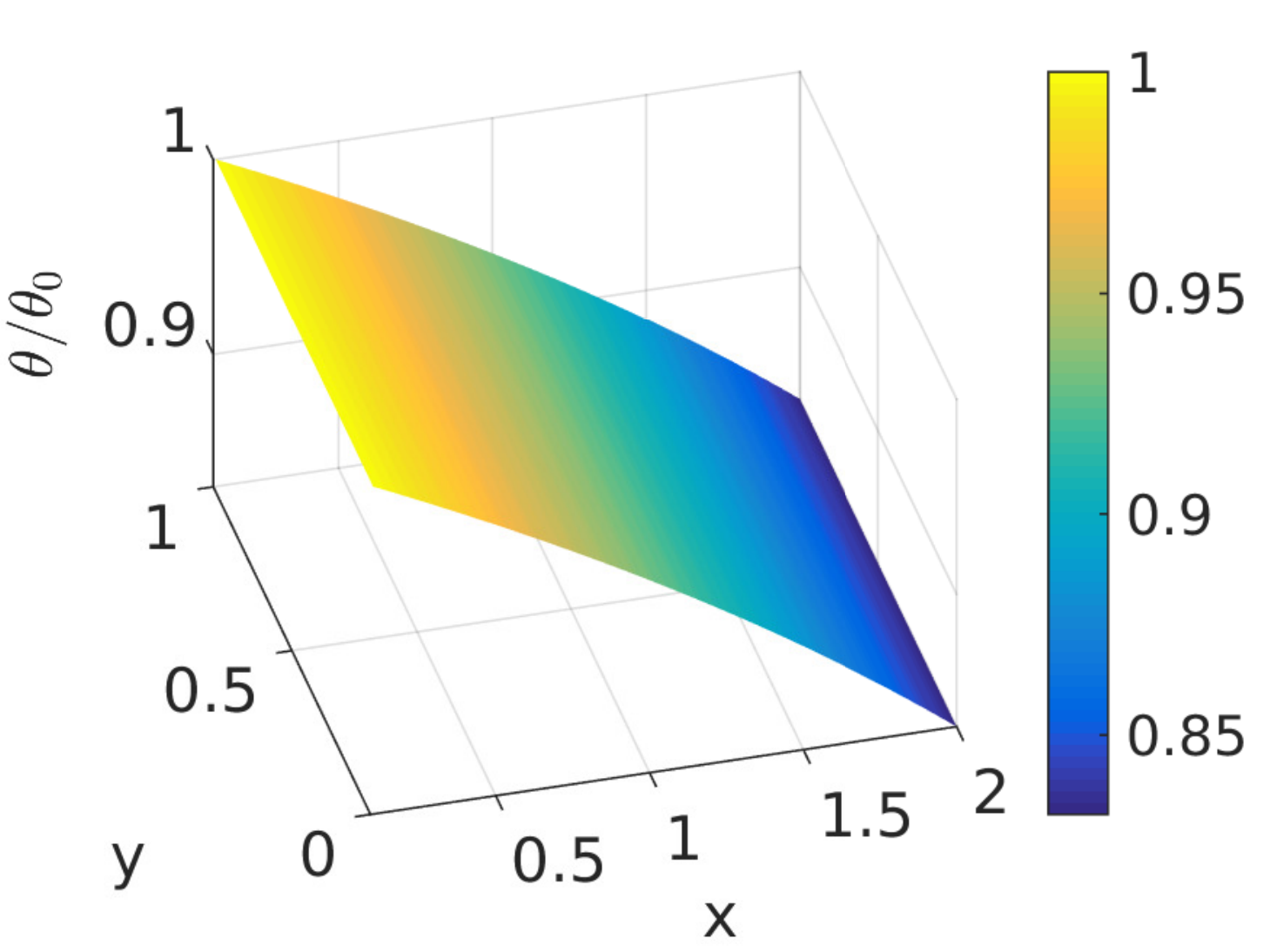}}
  
  \caption{Numerical solutions for the fluid velocity, the permeability and the porosity, at time $t = 300$, obtained using the network-inspired relation with $p_c = 0.3232$.} \label{Fig:Results_p_T}
\end{figure*}

\begin{figure*}[!ht]
  \centering
  \subfloat[Numerical solution for the fluid velocity.]{\includegraphics[width=0.3\textwidth]{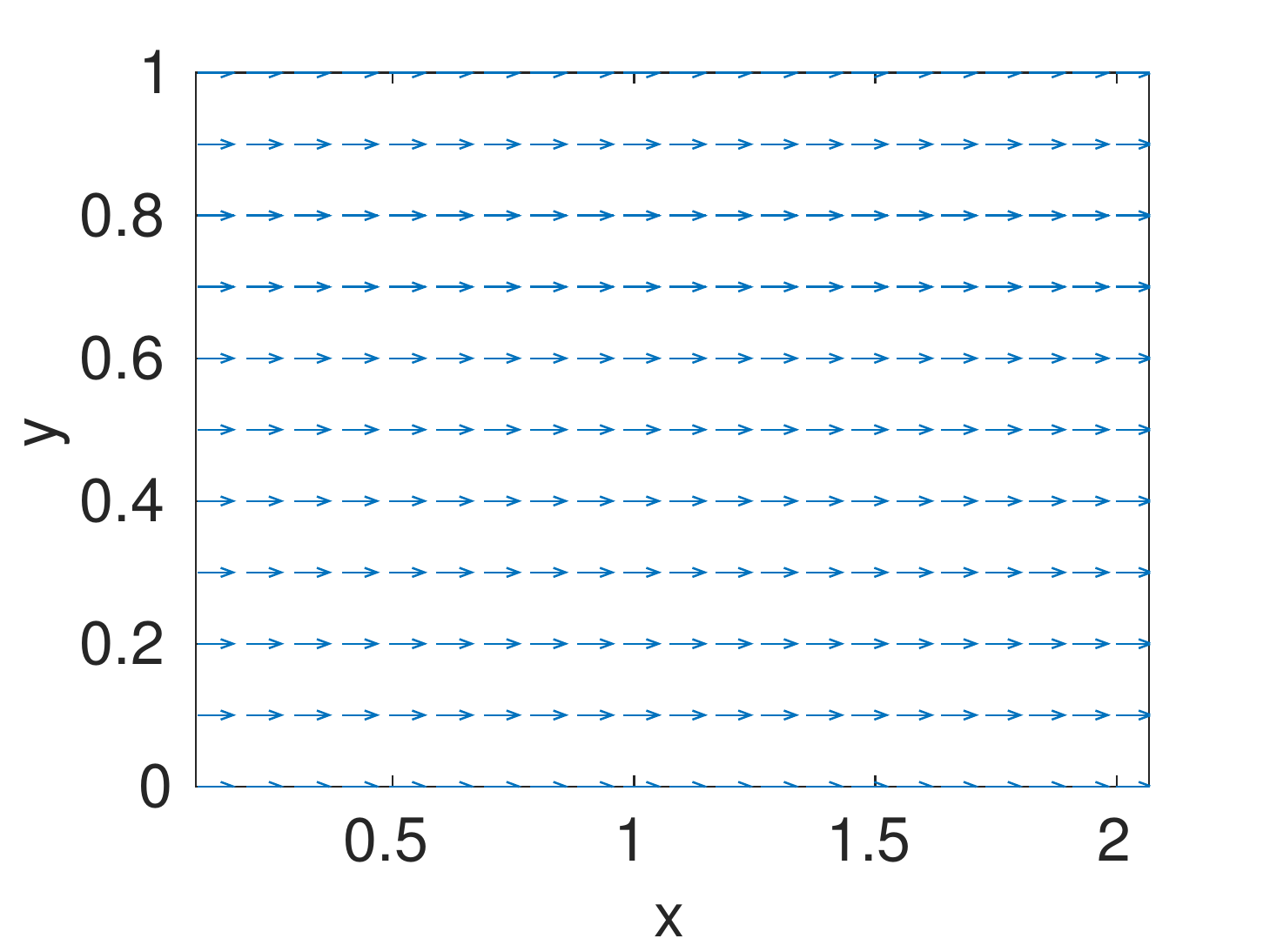}}%
  \qquad
  \subfloat[Numerical solution for the permeability. \label{Fig:Results_p_Q_kappa}]{\includegraphics[width=0.3\textwidth]{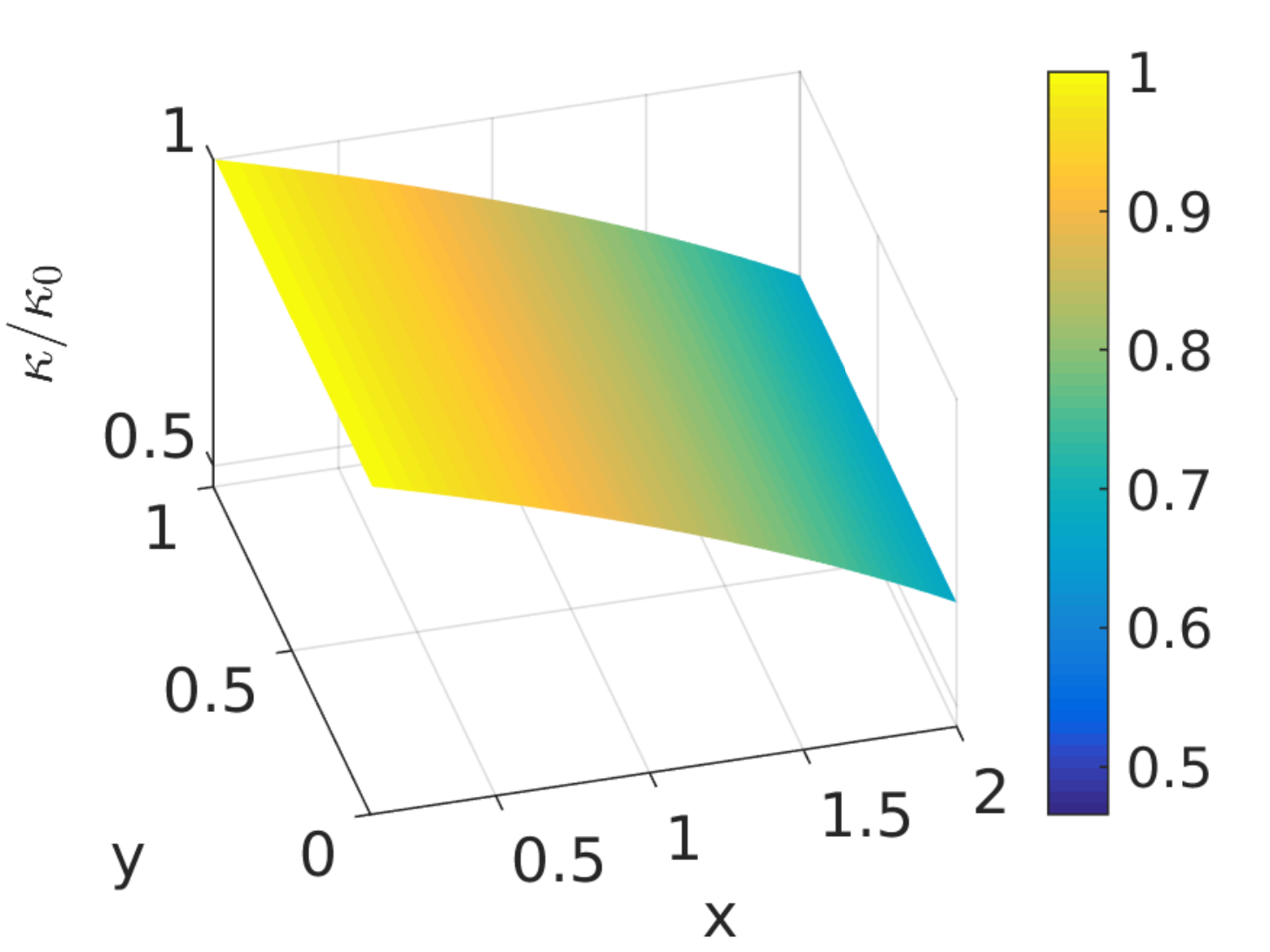}}%
  \qquad
  \subfloat[Numerical solution for the porosity. \label{Fig:Results_p_Q_theta}]{\includegraphics[width=0.3\textwidth]{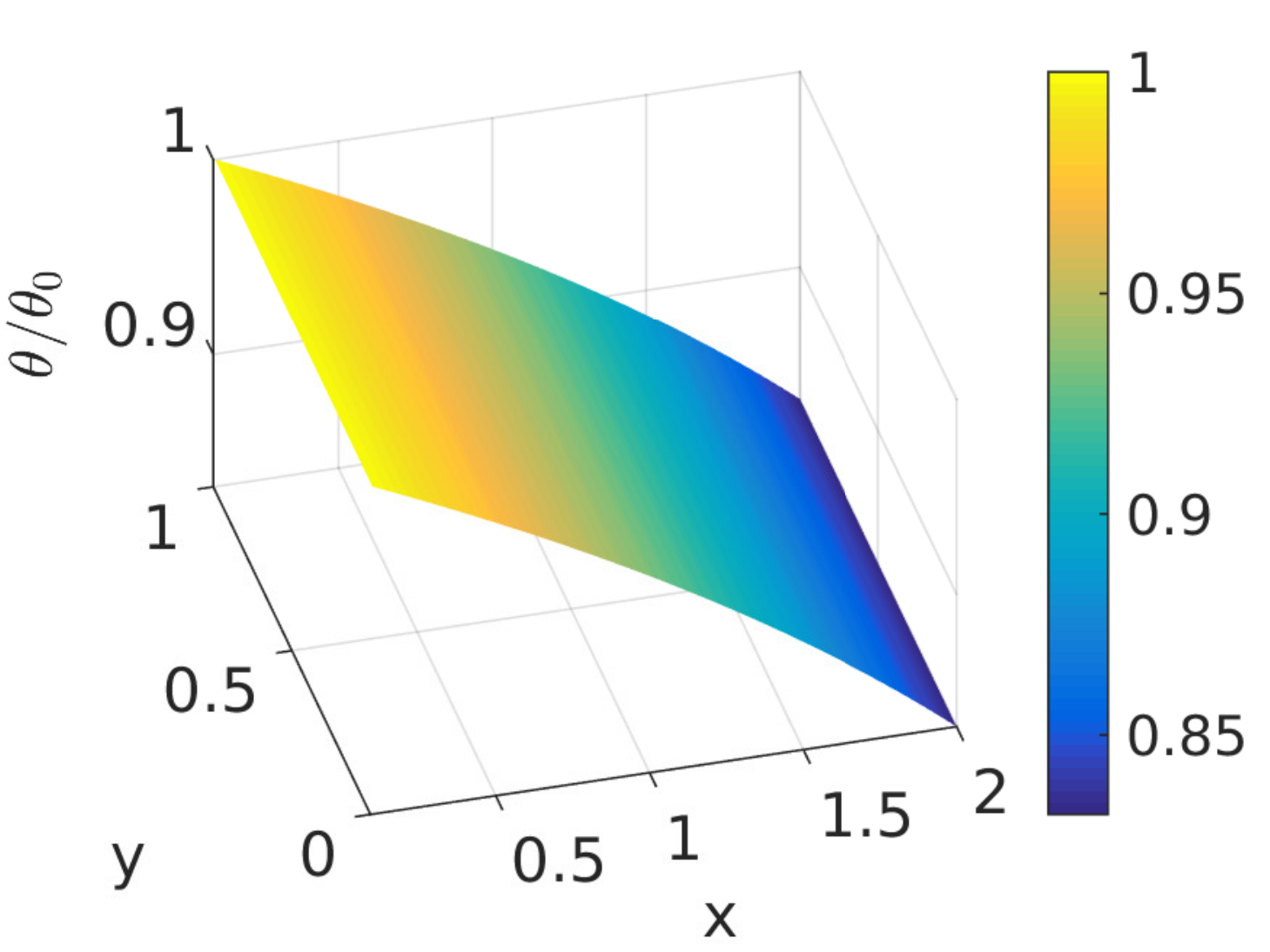}}
  
  \caption{Numerical solutions for the fluid velocity, the permeability and the porosity, at time $t = 300$, obtained using the network-inspired relation with $p_c = 0.4935$.} \label{Fig:Results_p_Q}
\end{figure*}

As shown in these figures, the injected water flows in the horizontal direction through the domain from the inlet to the outlet. Furthermore, the magnitude of the velocity $\abs{v}$ in the inlet is larger than in the outlet. Note that due to the continuity equation~\eqref{Eq:continuity}, the fluxes at the inlet and the outlet are not necessarily equal. The change in $\abs{v}$ can be explained by the permeability profiles shown in Figs.~\ref{Fig:Results_p_KC_kappa}-\ref{Fig:Results_p_Q_kappa}. In these figures we observe that the permeability decreases almost linearly from the inlet to the outlet. In addition, the permeability obtained using the Kozeny-Carman relation exhibits a larger decrease than the permeabilities obtained with the network-inspired model. The normalised permeability using the Kozeny-Carman relation decreases from 1 to 0.4659, while the normalised permeabilities for the network-inspired relation decrease from 1 to 0.7519 and from 1 to 0.6684 for the triangular structured and the rectangular network respectively. This behaviour is clarified by Fig.~\ref{fig_6.40} and Figs.~\ref{Fig:Results_p_KC_theta}-\ref{Fig:Results_p_Q_theta}. Due to boundary condition~\eqref{Eq:BC1_5}, the values of the normalised porosity in all three cases are almost the same in the outlet and are equal to 0.8321. In Fig.~\ref{fig_6.40}, we see that for this value of the porosity, the normalised permeability is the lowest for the Kozeny-Carman relation and the highest for the network-inspired relation derived from the triangular structured network. This explains the difference in decrease in the permeability profiles. Note that the values of the material properties in Table~\ref{Tab:MaterialValues} belong to a porous medium consisting of sand grains, hence a high pump pressure will apparently not result in hydraulic fracturing. This study therefore neglects this phenomenon.

In Fig.~\ref{Fig:Results_Q_p_fine}, the time average of the volumetric flow rate $\overline{Q}_{out}$ is depicted for different values of the percolation threshold. As expected from Fig.~\ref{fig_6.40}, for low percolation thresholds the network-inspired relation results in higher flow rates than the Kozeny-Carman relation. Furthermore, the flow rate changes significantly as a function of the percolation threshold. Hence the water flow depends on the topology of the network. The negative values of the flow rate $\overline{Q}_{out}$ for the very high values of the percolation threshold from the network-inspired relation are caused by violation of the M-matrix property that gives loss of monotonicity of the numerical solution. To demonstrate this, the problem is solved for a coarser grid $\Delta x = \Delta y = 0.04$, as shown in Fig.~\ref{Fig:Results_Q_p_coarse}. In this figure, it becomes even worse, and we observe spurious oscillations for the network-inspired relation for percolation thresholds larger than $0.85$ approximately. The reason for this behaviour is that for large values of the percolation threshold, the permeability goes to zero very soon. Therefore, even the used stabilisation did not diminish the nonphysical oscillations. Probably a stronger stabilisation will alleviate these spurious oscillations. This was behind the scope of the current paper.

\begin{figure*}[!ht]
  \centering
  \includegraphics[scale=0.2]{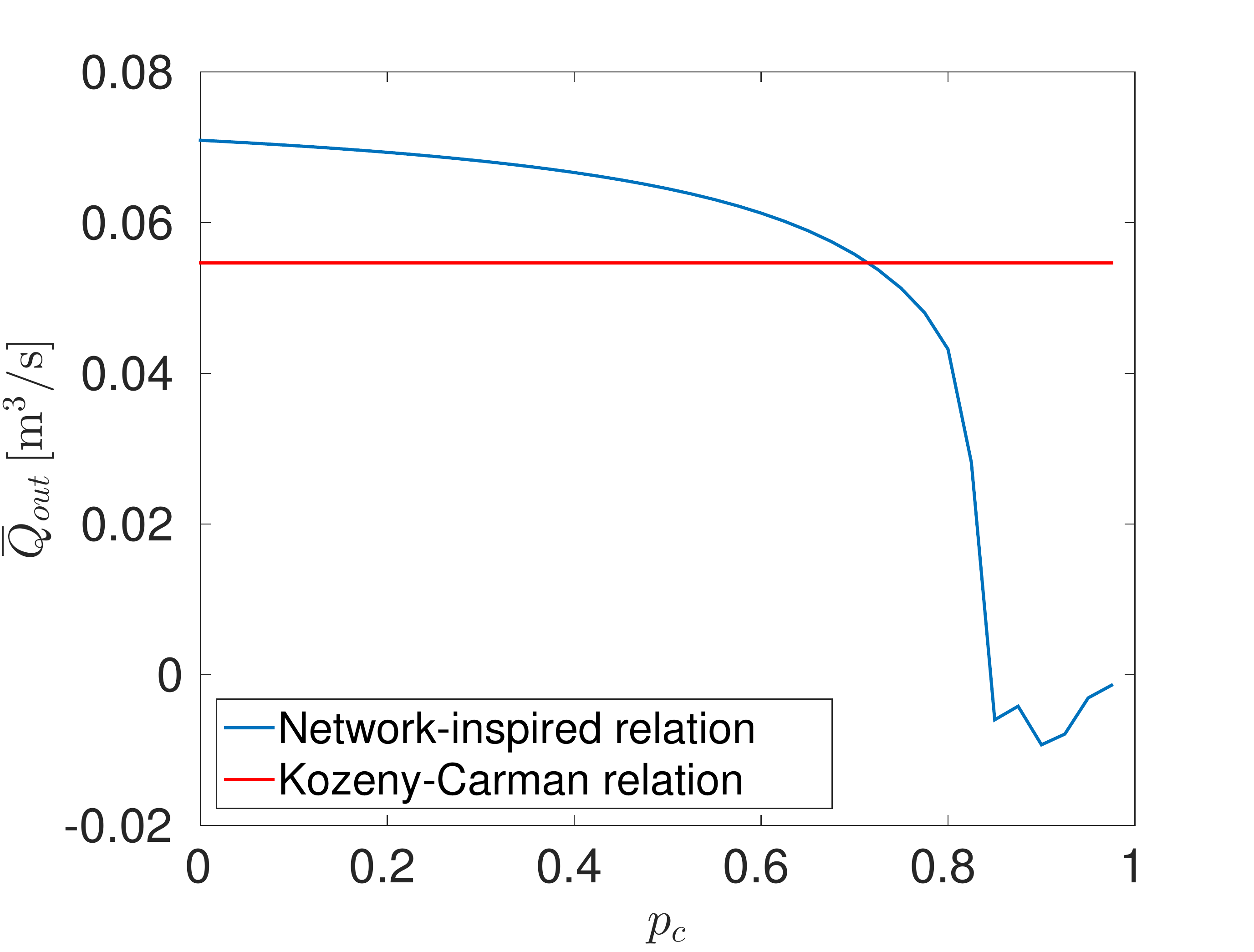}
  \caption{The time average of the volumetric flow rate $\overline{Q}_{out}$ as a function of the percolation threshold $p_c$, using $\Delta x = \Delta y = 0.02$.} \label{Fig:Results_Q_p_fine}
\end{figure*}

\begin{figure*}[!ht]
  \centering
  \includegraphics[scale=0.2]{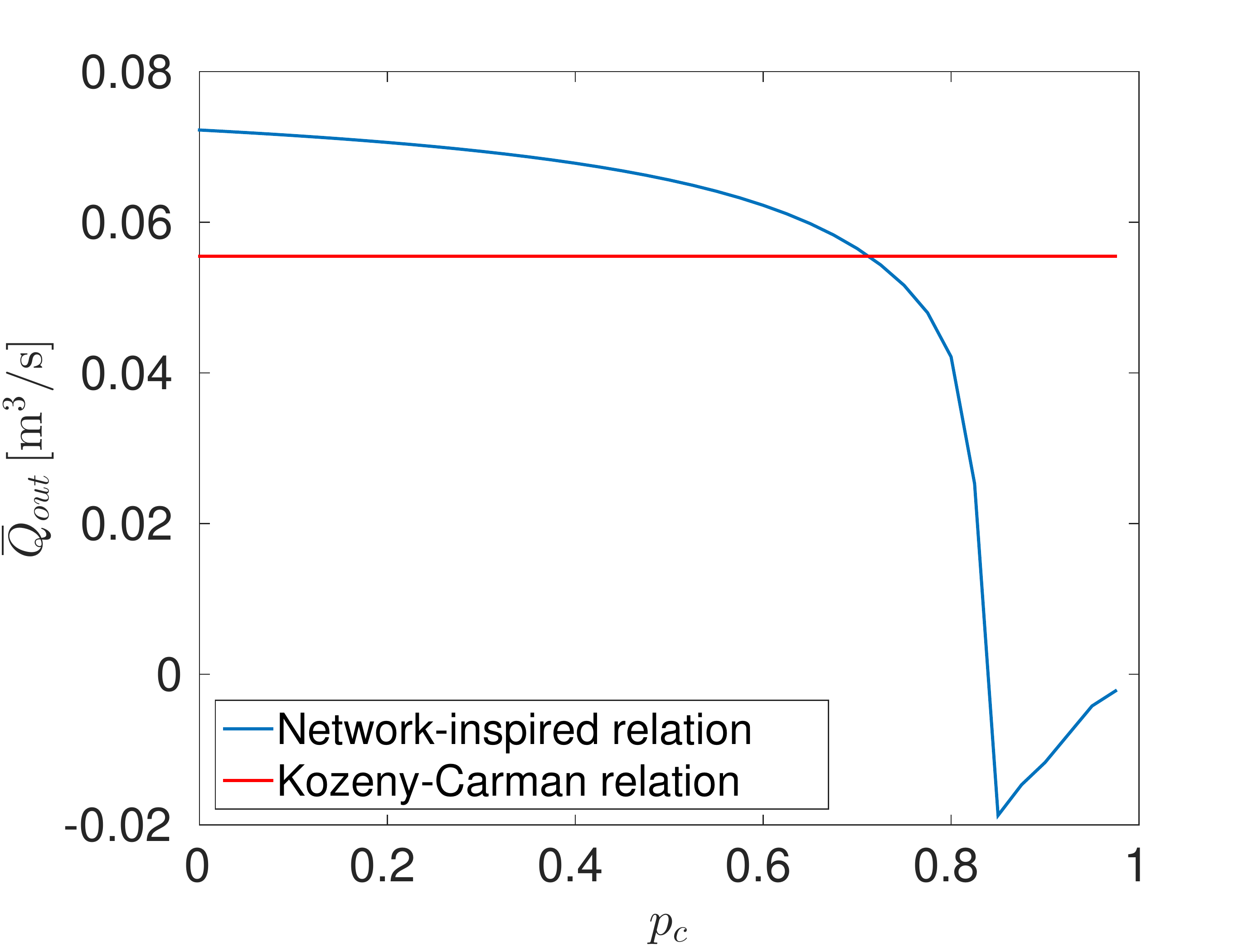}
  \caption{The time average of the volumetric flow rate $\overline{Q}_{out}$ as a function of the percolation threshold $p_c$, using $\Delta x = \Delta y = 0.04$.} \label{Fig:Results_Q_p_coarse}
\end{figure*}

\subsection{Numerical results for the squeeze problem}

The impact of the imposed vertical load on boundary segments~$\Gamma_1$ and~$\Gamma_5$ is shown in Figs.~\ref{Fig:Results_sigma_KC} -~\ref{Fig:Results_sigma_Q}, using the Kozeny-Carman relation and the network-inspired relation respectively. In these simulations, water is injected into the porous medium at a constant pump pressure equal to $5.0\, \mathrm{bar}$. The simulated fluid velocity, permeability and porosity profiles that are obtained using the Kozeny-Carman relation are provided in Fig.~\ref{Fig:Results_sigma_KC}, while the simulated results that are obtained using the network-inspired relation with $p_c = 0.3232$, corresponding with a triangular structured network, are provided in Fig.~\ref{Fig:Results_sigma_T}.  In Fig.~\ref{Fig:Results_sigma_Q}, the simulated results that are obtained using the network-inspired relation with $p_c = 0.4935$, corresponding with a rectangular network, are depicted.

\begin{figure*}[!ht]
  \centering
  \subfloat[Numerical solution for the fluid velocity. \label{Fig:Results_sigma_KC_v}]{\includegraphics[width=0.3\textwidth]{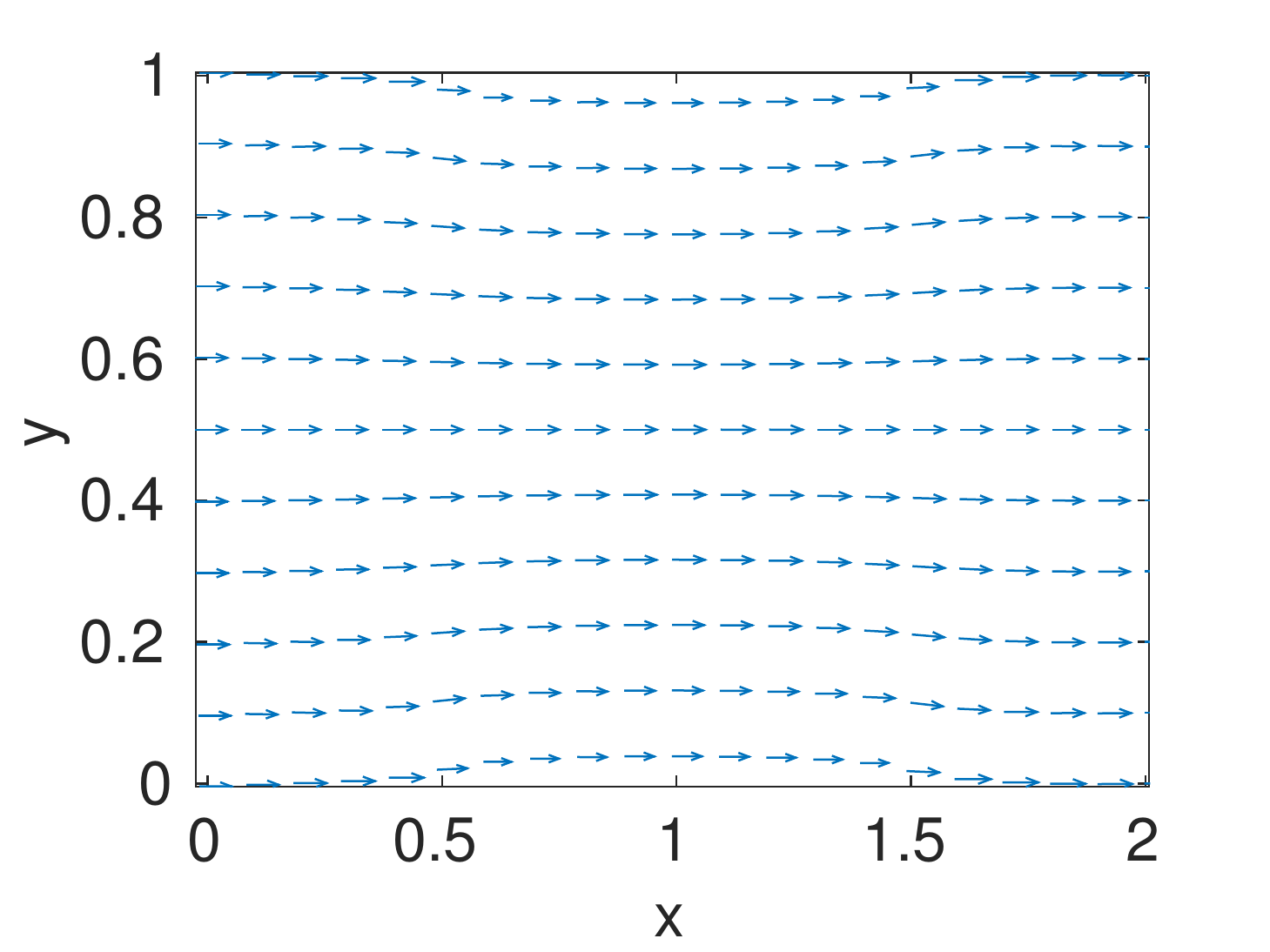}}%
  \qquad
  \subfloat[Numerical solution for the permeability. \label{Fig:Results_sigma_KC_kappa}]{\includegraphics[width=0.3\textwidth]{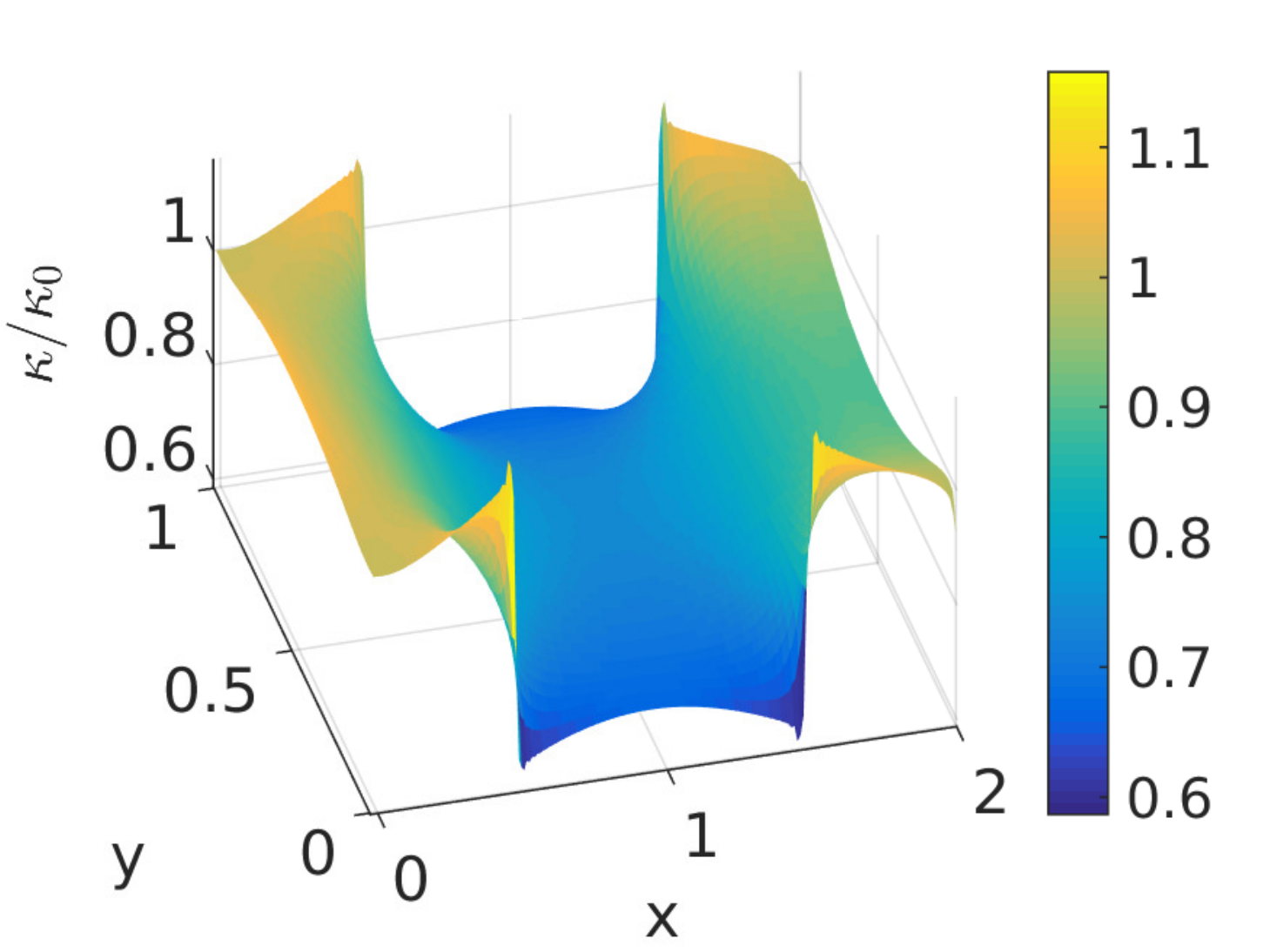}}%
  \qquad
  \subfloat[Numerical solution for the porosity.]{\includegraphics[width=0.3\textwidth]{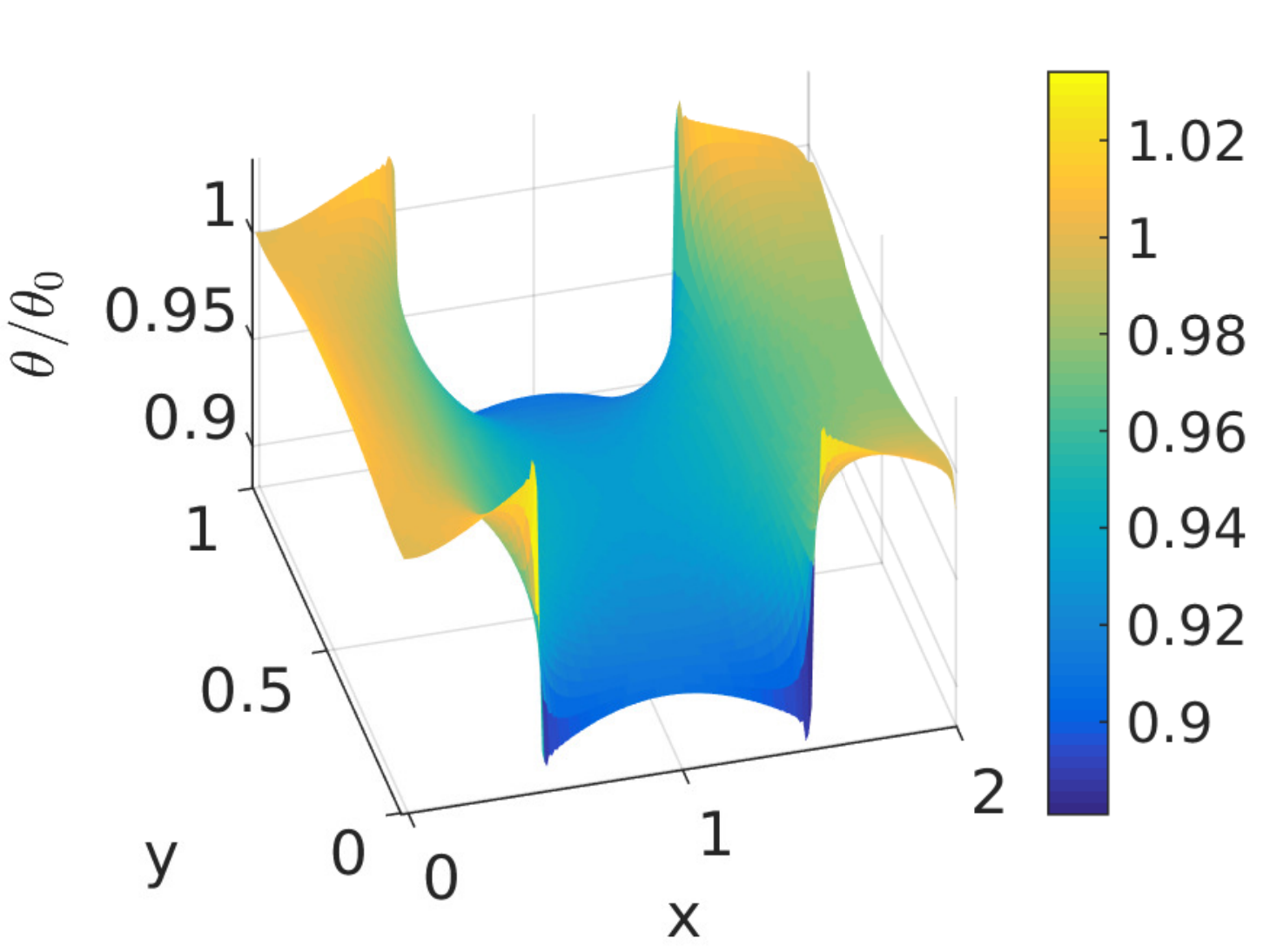}}
  
  \caption{Numerical solutions for the fluid velocity, the permeability and the porosity, at time $t = 300$, obtained using the Kozeny-Carman relation.} \label{Fig:Results_sigma_KC}
\end{figure*}

\begin{figure*}[!ht]
  \centering
  \subfloat[Numerical solution for the fluid velocity. \label{Fig:Results_sigma_T_v}]{\includegraphics[width=0.3\textwidth]{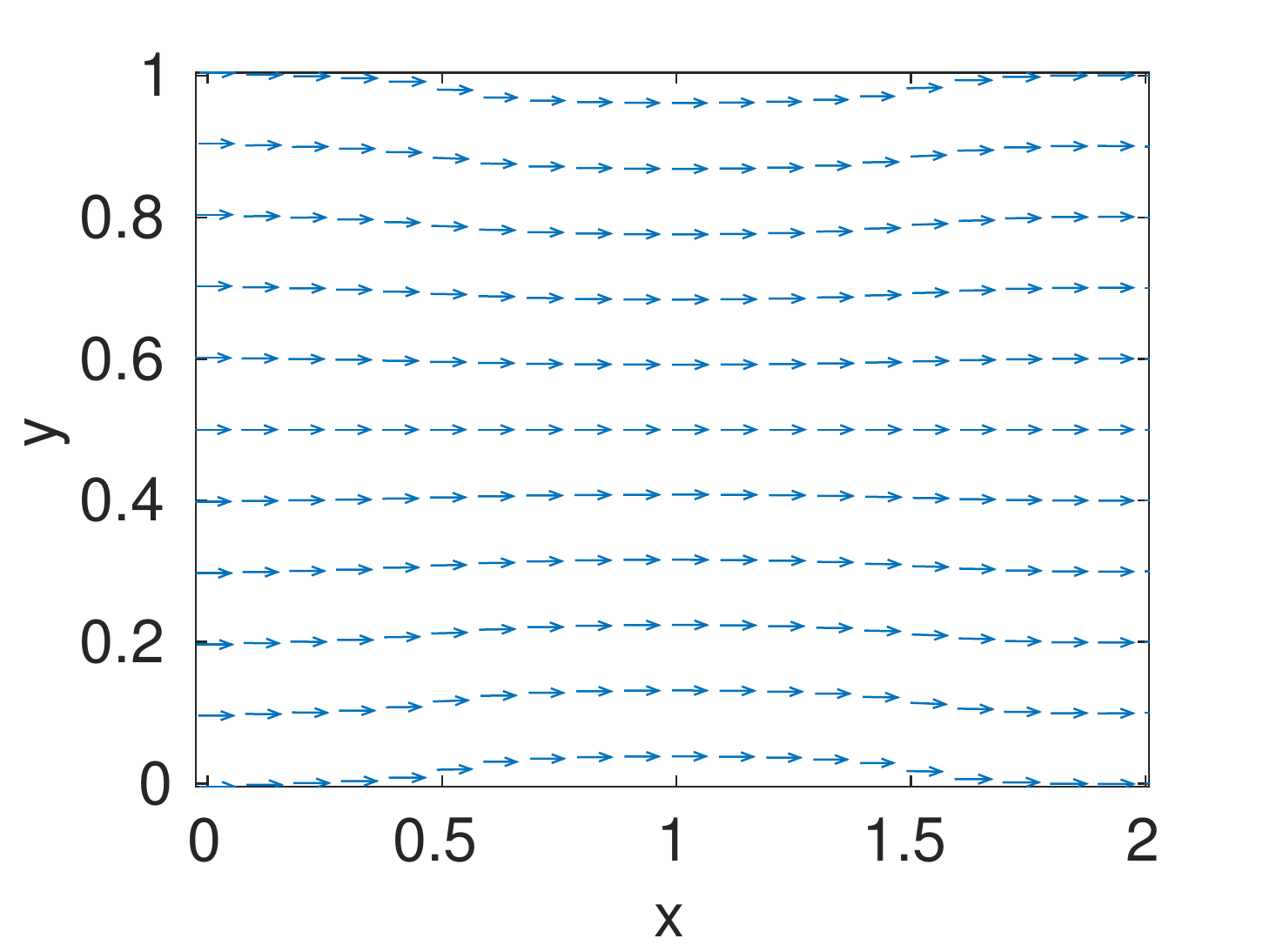}}%
  \qquad
  \subfloat[Numerical solution for the permeability. \label{Fig:Results_sigma_T_kappa}]{\includegraphics[width=0.3\textwidth]{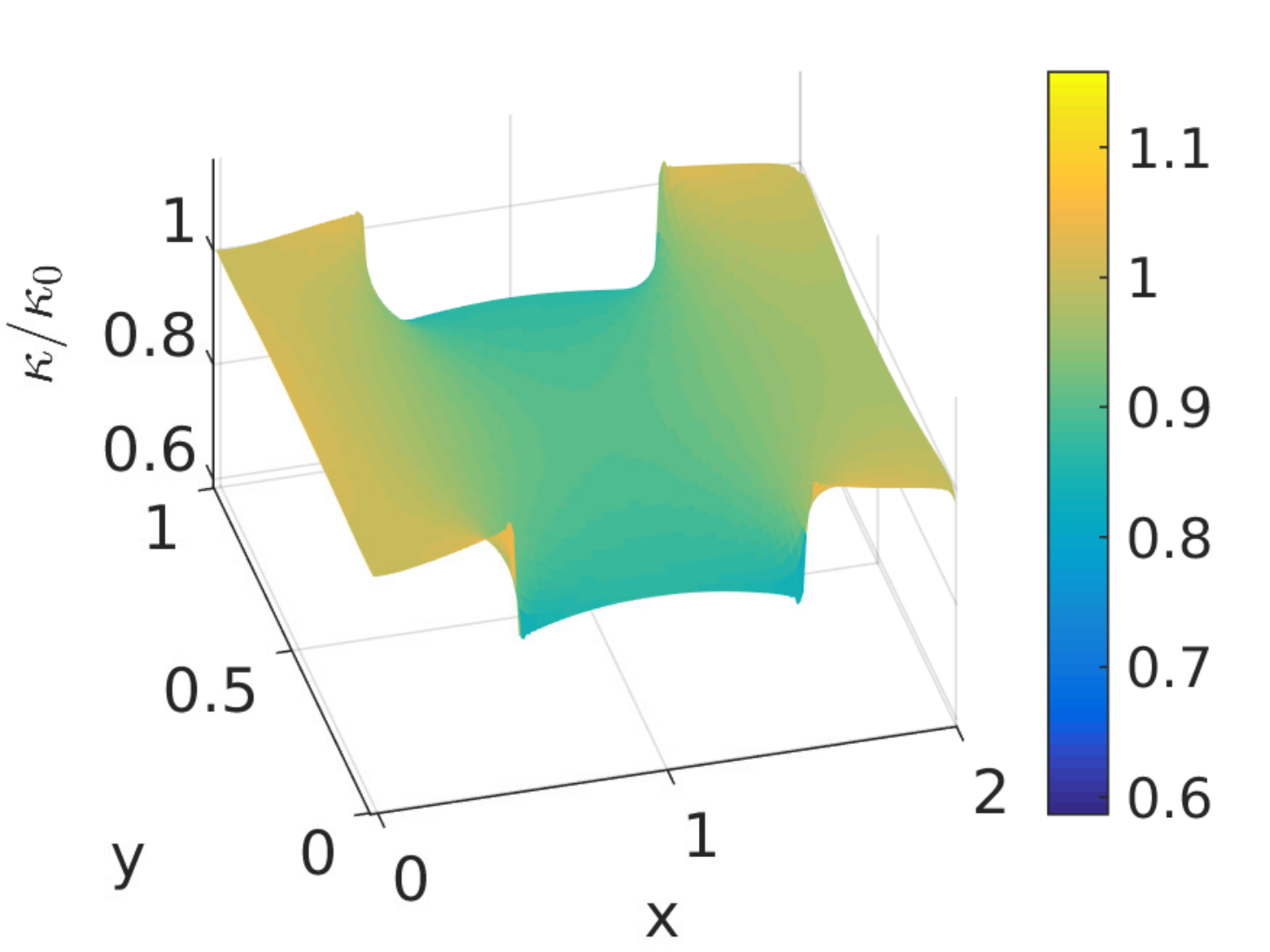}}%
  \qquad
  \subfloat[Numerical solution for the porosity.]{\includegraphics[width=0.3\textwidth]{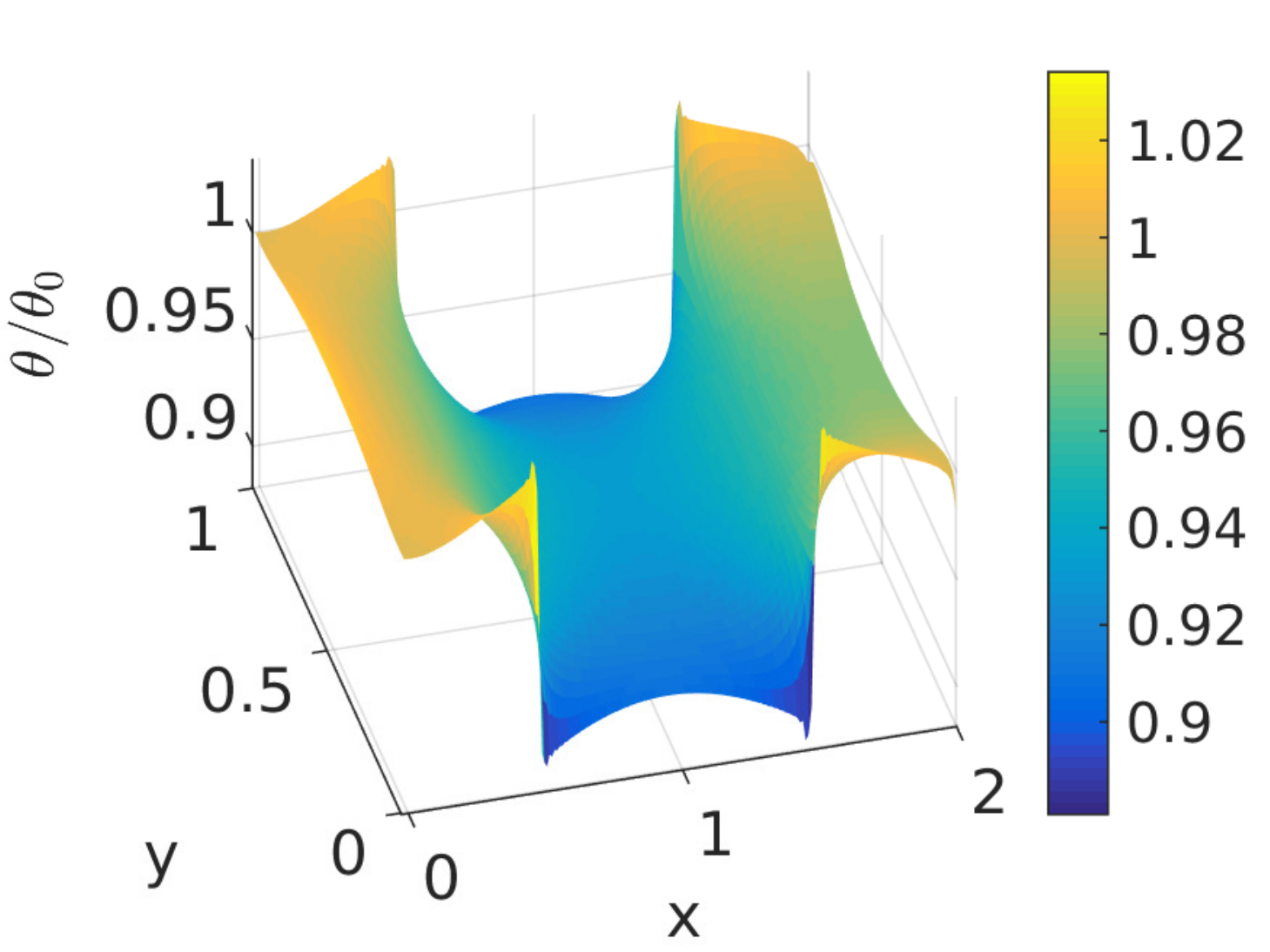}}
  
  \caption{Numerical solutions for the fluid velocity, the permeability and the porosity, at time $t = 300$, obtained using the network-inspired relation with $p_c = 0.3232$.} \label{Fig:Results_sigma_T}
\end{figure*}
 
\begin{figure*}[!ht]
  \centering
  \subfloat[Numerical solution for the fluid velocity. \label{Fig:Results_sigma_Q_v}]{\includegraphics[width=0.3\textwidth]{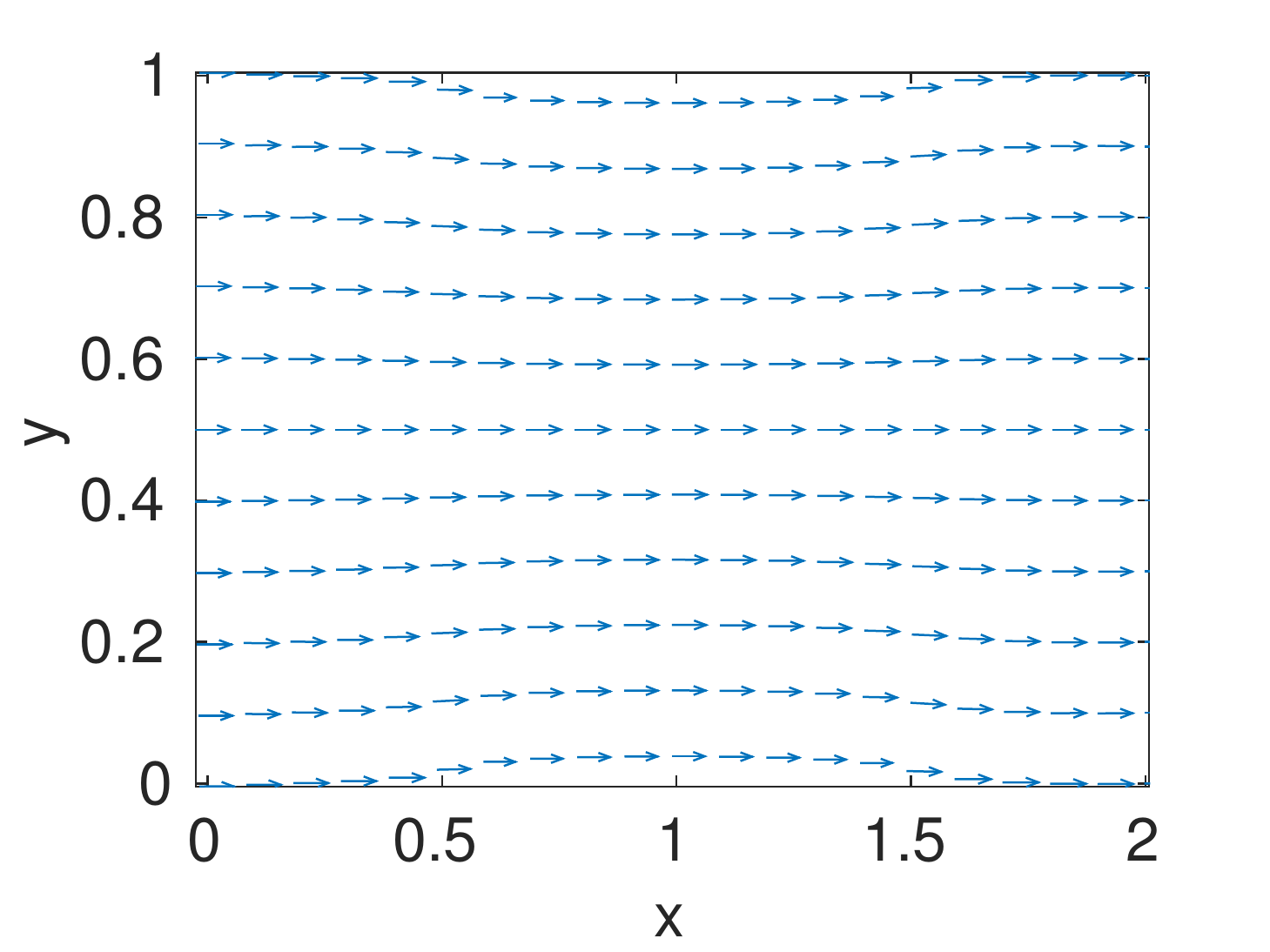}}%
  \qquad
  \subfloat[Numerical solution for the permeability. \label{Fig:Results_sigma_Q_kappa}]{\includegraphics[width=0.3\textwidth]{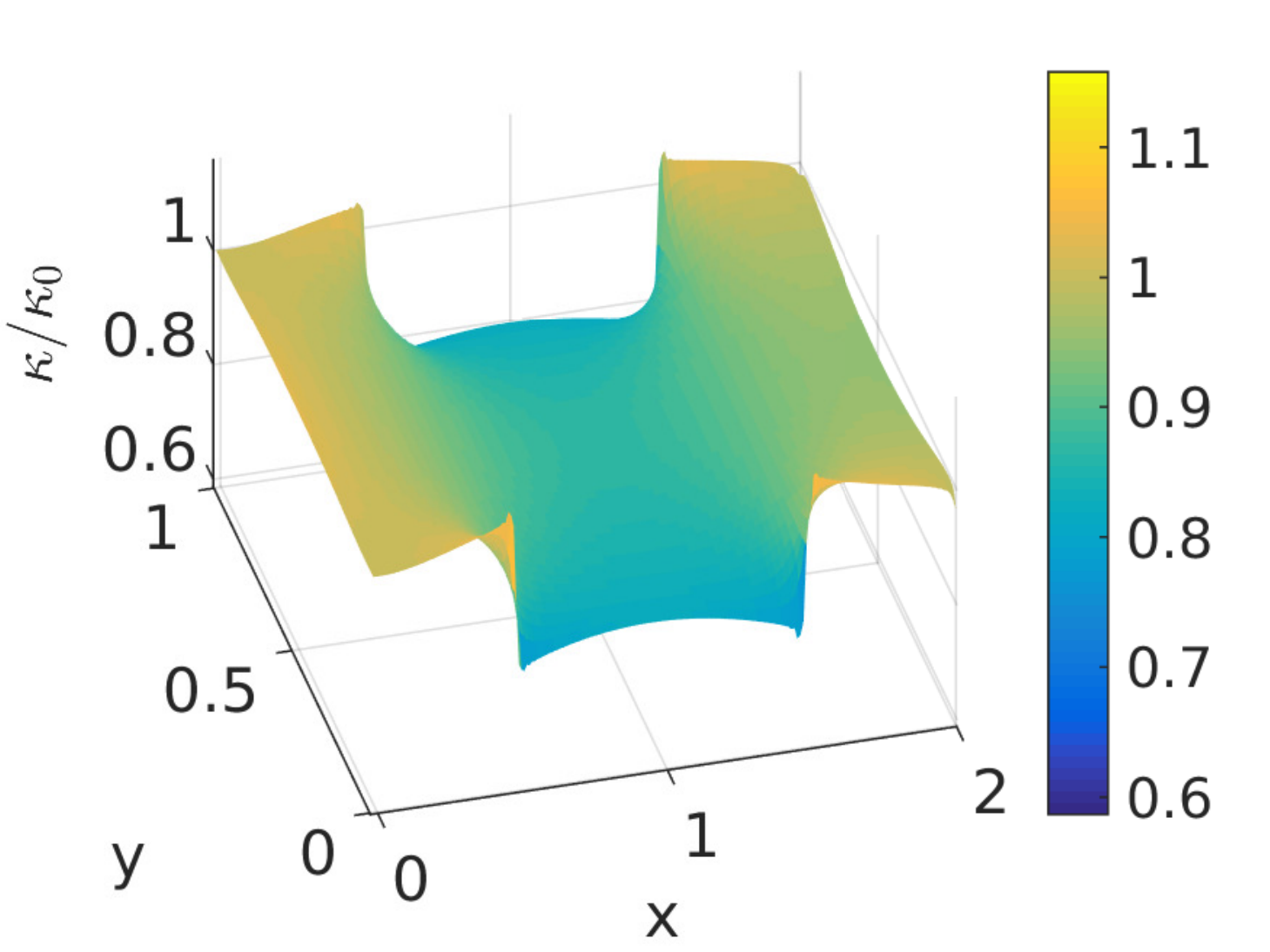}}%
  \qquad
  \subfloat[Numerical solution for the porosity.]{\includegraphics[width=0.3\textwidth]{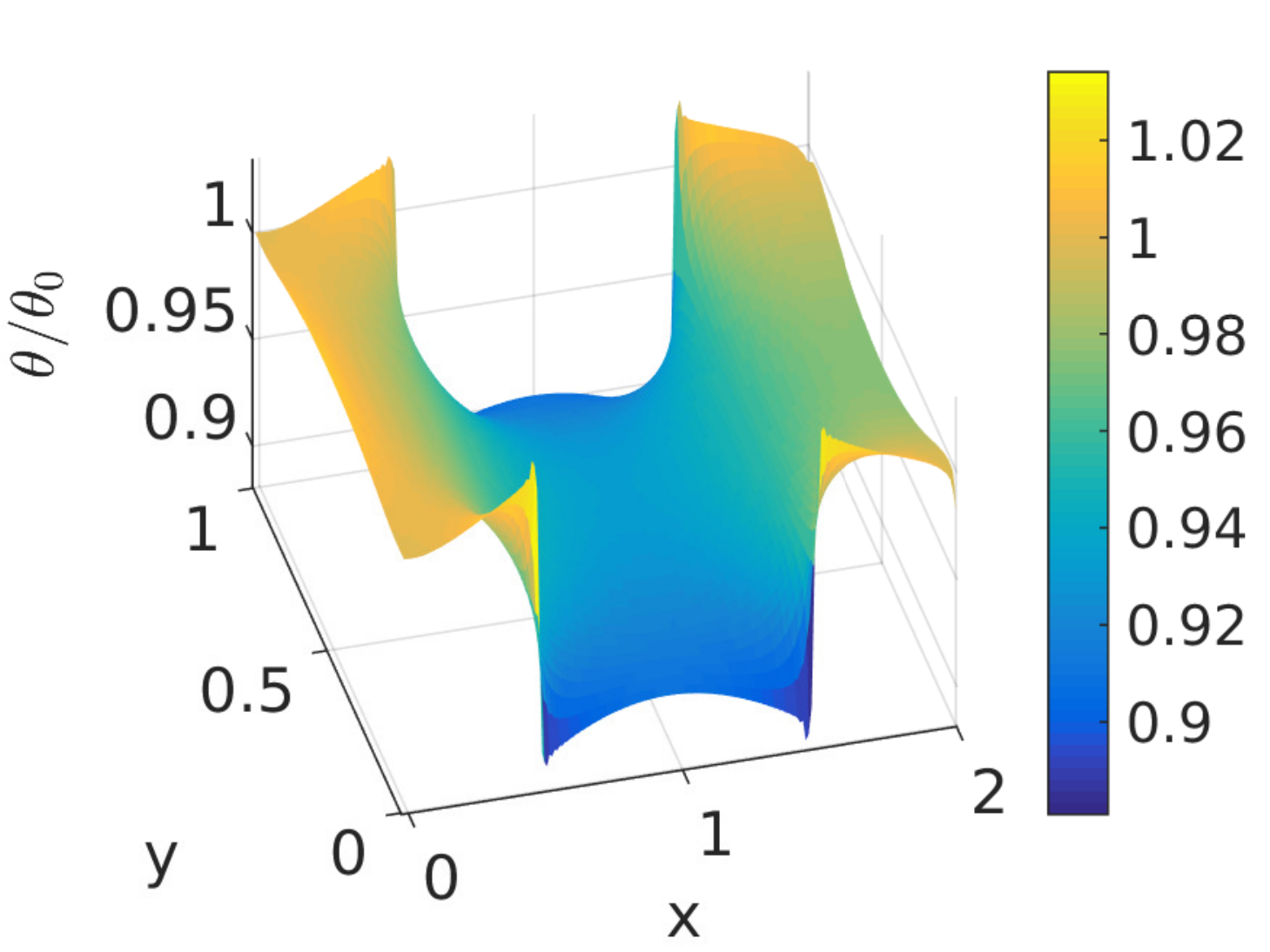}}
  
  \caption{Numerical solutions for the fluid velocity, the permeability and the porosity, at time $t = 300$, obtained using the network-inspired relation with $p_c = 0.4935$.} \label{Fig:Results_sigma_Q}
\end{figure*}

In Figs.~\ref{Fig:Results_sigma_KC_v}-\ref{Fig:Results_sigma_Q_v}, the impact of the imposed vertical load on the computational domain is shown. The magnitude of the velocity is small near the boundary segments where the vertical load is applied and large in the middle near the symmetry axis $y = H/2$. In the outlet, the magnitude of the velocity is maximal near the upper and lower boundary segments and decreases towards the symmetry axis. In all three cases, the fluid flows mainly in the horizontal direction. In the region where the load is imposed, the domain is squeezed, resulting in a larger density of the grains. This leads to a lower porosity in this region, with minimum values 0.8809, 0.8811 and 0.8810 for the Kozeny-Carman relation, the triangular structured network-inspired relation and the rectangular network-inspired relation respectively. The abrupt transition in the boundary condition between boundary segments $\Gamma_1$ and $\Gamma_2$ (and between $\Gamma_4$ and $\Gamma_5$), which results in a discontinuous force, results in a small numerical artefact at the location of these transitions. As expected from Fig.~\ref{fig_6.40} and the minimum values for the porosities, the decrease in permeability by the Kozeny-Carman relation, as shown in Fig.~\ref{Fig:Results_sigma_KC_kappa}, is larger than the decrease in the porosities obtained using the network-inspired relation, Figs.~\ref{Fig:Results_sigma_T_kappa} and~\ref{Fig:Results_sigma_Q_kappa}.

In Fig.~\ref{Fig:Results_Q_sigma}, the time average of the volumetric flow rate $\overline{Q}_{out}$ is depicted for different values of the percolation threshold. Similarly to the high pump pressure problem, the flow rates for small values of the percolation threshold using the network-inspired relation are higher than the flow rates obtained using the Kozeny-Carman relation. In addition, we observe that the flow rate depends significantly on the percolation threshold and hence on the topology of the network for large percolation thresholds.

\begin{figure*}[!ht]
  \centering
  \includegraphics[scale=0.2]{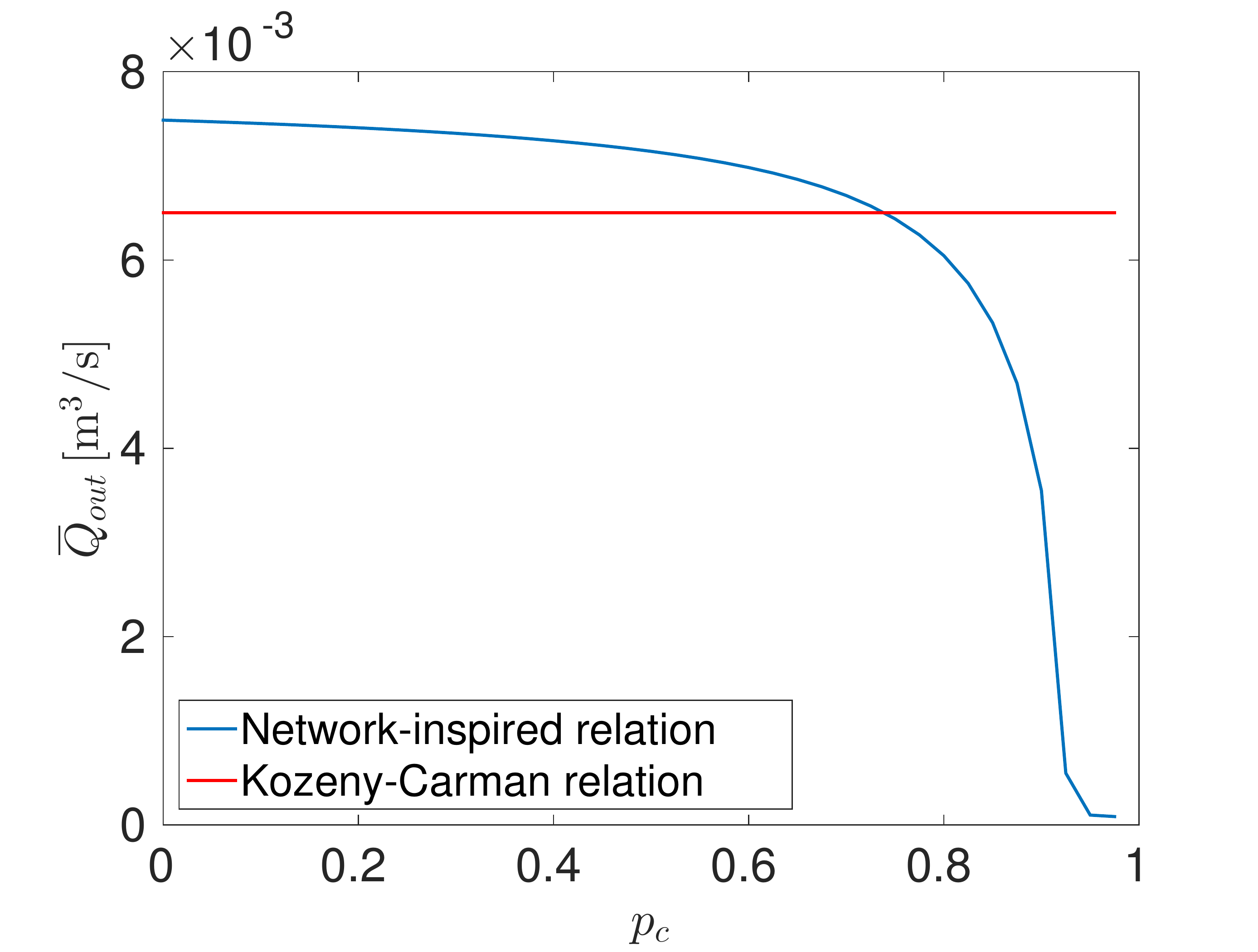}
  \caption{The time average of the volumetric flow rate $\overline{Q}_{out}$ as a function of the percolation threshold $p_c$.} \label{Fig:Results_Q_sigma}
\end{figure*}


\section{Discussion and conclusions}

In this paper, the network-inspired permeability-porosity relation is applied on two poroelasticity problems. This numerical experiment is designed in order to analyse the applicability of this microscopic relation on the macro-scale. Furthermore, we compare the results obtained with the network-inspired relation to the Kozeny-Carman relation which is often used in these physical problems. In the first problem, a high pump pressure is imposed in the inlet of a porous medium package. This high pressure forces the grains to move towards the outlet. In the second problem the package is squeezed by applying a load on the middle of the top and bottom edges of the domain. The purpose of considering these poroelasticity problems is to create a large density of the grains in the computational domain which results in a decrease of the porosity. In these problems, Biot's model for poroelasticity is used to determine the water pressure and the displacements of the grains that are needed to compute the porosity. From the porosity the permeability is determined either by the network-inspired relation or by the Kozeny-Carman relation. Depending on the topology, three different percolation thresholds, corresponding with a rectangular network ($p_c = 0.4935$), triangular structured network ($p_c = 0.3232$) and triangular unstructured network ($p_c = 0.3438$), are distinguished. However, since the topology of macro-scale porous media is not known, computations are also performed with percolation thresholds in the interval $[0, 0.975]$ to investigate the influence of the percolation threshold (and hence the topology of the porous medium) on the flow rate.

First, the problems are solved with the Kozeny-Carman relation, the network-inspired relation based on the triangular structured network and the relation based on the rectangular network. From the numerical results we conclude that the permeability obtained using the Kozeny-Carman relation exhibits a larger decrease than the permeabilities obtained with the network-inspired relations, which is clarified by Fig.~\ref{fig_6.40}. In contrast, the porosity profile is not affected significantly by the selected permeability-porosity relation. Second, the time average of the volumetric flow rate was computed for percolation thresholds in the interval $[0, 0.975]$. For low percolation thresholds the network-inspired relation results in higher flow rates than the Kozeny-Carman relation, as expected from Fig.~\ref{fig_6.40}. In addition, it is shown that the flow rate changes significantly as a function of the percolation threshold which means that the water flow depends on the topology of the network. For large percolation thresholds, spurious oscillations appeared due to the violation of the M-matrix property in the discretisation matrix that resulted from the convergence of Biot's problem to the related saddle point problem, as proven in Sect.~\ref{Sec:Proof}. The results for these percolation thresholds could be improved by using a finer grid.

For the studied problems and the set of parameters chosen, we noticed that the applied permeability-porosity relations result in small changes in the porosity while a major change is realised in the permeability profiles. A possible explanation for this behaviour is that the relation between the velocity field and the change of the displacements in time as stated in Eq.~\eqref{Eq:continuity}, is not strong enough to lead to significant changes in the porosity profile.

\paragraph{Acknowledgements} The work of M. Rahrah was supported by the Netherlands Organisation for Scientific Research NWO (project number 13263). The work of L.A. Lopez-Pe\~na was supported by the Mexican Institute of Petroleum (IMP) through the Programa de Captaci\'on de Talento, Reclutamiento, Evaluaci\'on y Selecci\'on de Recursos Humanos (PCTRES) grant.

\bibliographystyle{abbrv}
\bibliography{my_bib_PorPer}

\end{document}